\documentclass[11pt]{article}

\usepackage[utf8]{inputenc}
\usepackage{amsthm}
 \usepackage{amssymb}
 \usepackage{amsmath}
 \usepackage{hyperref}

\usepackage{comment}

\usepackage{graphicx}
\usepackage{color}
\usepackage{tikz}
\usepackage{float}
\usepackage{subfig}

\numberwithin{equation}{section}

 \hypersetup{colorlinks,linkcolor={blue},citecolor={red}}
 \usepackage{verbatim}
 \usepackage{appendix}

\newtheorem{thm}{Theorem}[section]
\newtheorem{conj}{Conjecture}[section]
\newtheorem{prop}[thm]{Proposition}
\newtheorem{lem}[thm]{Lemma}
\newtheorem{cor}[thm]{Corollary}

\newtheorem{claim}[thm]{Claim}
\theoremstyle{definition}
\newtheorem{defi}[thm]{Definition}

\newtheorem{prob}[thm]{Problem}
\theoremstyle{remark}
\newtheorem{rmk}[thm]{Remark}
\numberwithin{equation}{section}
\newcommand{\be}{\beta}
\newcommand{\beq}{\begin{equation}}
\newcommand{\eeq}{\end{equation}}
\newcommand{\PP}{\mathbb{P}}

\newcommand{\K}{K\"{a}hler}    
\newcommand{\KE}{K\"{a}hler--Einstein\;}

\def\KEE{K\"ahler--Einstein edge }

 \newcommand{\NN}{{\mathbb N}}
\newcommand{\FF}{{\mathbb F}}

\title{Eguchi--Hanson metrics arising from\\ K\"ahler-Einstein edge metrics}

\date{14 December 2023}

\begin{document}
\author{Yuxiang Ji, Yanir A. Rubinstein, Kewei Zhang}
\maketitle

\centerline{\it Dedicated to Scott Wolpert on the occassion of his retirement }


\begin{abstract}
    Calabi--Hirzebruch manifolds are higher-dimensional generalizations of both the football and Hirzebruch surfaces. 
    We construct a family of \KE edge metrics singular along two disjoint divisors on the Calabi--Hirzebruch manifolds and study their Gromov--Hausdorff limits when either cone angle tends to its extreme value. As a very special case, we show that the
    celebrated Eguchi--Hanson metric arises in this way naturally as a  Gromov--Hausdorff limit. We also completely describe all other (possibly rescaled) Gromov--Hausdorff limits which exhibit a wide range of behaviors, resolving in this setting a conjecture of Cheltsov--Rubinstein. This gives a new interpretation of
    both the Eguchi--Hanson space
    and Calabi's Ricci flat spaces
    as limits of compact singular Einstein spaces.
\end{abstract}

\tableofcontents


\section{Motivation}

The main motivation for this
work is a program of Cheltsov--Rubinstein 
concerning the small angle deformation of \KEE (KEE) metrics,
which in particular makes the following prediction \cite[Conjecture 1.11]{CR15}:

\begin{conj}
\label{GeneralConj}
Suppose that $(X,D)$ is strongly asymptotically log
Fano manifold with $D$ smooth and irreducible.
Suppose that 
$
\kappa:=\inf\{\NN\ni k\le\dim X \,:\, (K_X+D)^k=0\}\le\dim X,
$
and that there exist KEE metrics $\omega_\be, \be\in(0,\epsilon)$ on $(X,D)$
for some $\epsilon>0$.
Then, $(X,D,\omega_\be)$
converges in an appropriate sense as $\be$ tends to zero  to a
generalized KE metric $\omega_\infty$ 
that is Calabi--Yau along its generic $(\dim X+1-\kappa)$-dimensional
fibers.

\end{conj}

In this article 
we actually treat a slightly more general situation where $D$ is allowed
to have two disjoint smooth components $D=D_1+D_2$, but the setting is
essentially identical to that of Conjecture \ref{GeneralConj} since
the angle $2\pi\beta_1$ along $D_1$ of the KEE metric $\omega_{\be_1,\be_2}$ 
actually determines the angle  $2\pi\beta_2$ along $D_2$
and, importantly, vice versa.

A second motivation for this article is given by a prediction posed by two of the present authors in a previous work
concerning the {\it large} angle limit of a family of 
\KEE metrics constructed on the second Hirzebruch surface $\FF_2$. On that surface let $D_1:=Z_{-2}$ denote the $-2$-curve and $D_2:=Z_2$ the smooth
infinity section, a $2$-curve
satisfying $Z_{-2}\cap Z_2=\emptyset$.
According to \cite[Theorem 1.2]{RZ21} there exist for each $\beta_1\in(0,1)$
a unique \KEE metric $\omega_{\be_1,\be_2}$  with angle $2\pi\beta_1$ along $Z_{-2}$ and angle
$2\pi\beta_2=2\pi(2\be_1-3+\sqrt{9+12\beta_1-12\beta_1^2})/4$ along $Z_2$
and cohomologous to $\frac{1+\be_2}{1-\be_1}[Z_2]-[Z_{-2}]$.
The following prediction was made \cite[Remark 5.1]{RZ21}:

\begin{conj}
\label{RZConj}
As $\beta_1$ tends to $1$, an appropriate limit of
$(\FF_2,\omega_{\be_1,\be_2})$ converges to the Eguchi--Hanson metric.
\end{conj}

The virtue of Conjecture \ref{RZConj} is that it proposes
a remarkable new interpretation of the Eguchi--Hanson space
from mathematical physics
as an angle deformation limit of compact singular Einstein spaces with edge singularities.

The goal of the present article is to treat both conjectures
and a bit more.
We solve Conjecture  \ref{GeneralConj} in the 
setting of Calabi--Hirzebruch manifolds as well as solve  
Conjecture \ref{RZConj}. In fact we solve a
generalized version of Conjecture \ref{RZConj}
that interprets Calabi's Ricci flat spaces
as limits of compact KEE spaces.
Moreover,
due to the existence of
{\it two} disjoint divisors 
there are also interesting limits
to study that are not part of
the above conjectures:
these limits
are studied in Theorems
\ref{thm: main} and 
\ref{thm: main relation} below.

\section{Results}

Let $M$ be a compact K\"{a}hler manifold and $D=D_1+\cdots D_r$ a simple normal crossing divisor in $M$. A K\"{a}hler metric $\omega$ is said to have edge singularity along $D$ if $\omega$ is smooth on $M\setminus D$ and asymptotically equivalent to the model edge metric along $D$ \cite[Definition 3.1]{Y14}. The study of \KE edge metrics dates back to Tian \cite{Tian96} where he considered applications of such metrics to algebraic geometry. Cheltsov--Rubinstein \cite{CR15} initiated the program of studying small angle limits of \KE edge metrics. In previous works, two of us treated the Riemann surface footballs case and Hirzebruch surfaces case by first constructing \KE edge metrics on the manifolds using Calabi ansatz
and then studying their limiting behaviors when the cone angles tend to $0$
\cite{RZ21, RZ20}. 
In this paper, we consider a more general setting, Calabi--Hirzebruch manifolds.
To construct these KEE metrics we use
the standard Calabi ansatz
in Section \ref{sec: general}, generalizing \cite{RZ21}. 
The angle at either divisor
$D_1:=Z_{n,k}$ or $D_2:=Z_{n,-k}$
then determines the angle
on the other divisor which leads to two families of KEE
metrics $\eta_{\beta_1}$ and $\xi_{\beta_2}$
on $\mathbb{F}_{n, k}$.

\begin{rmk}
Very recently,
Biquard--Guenancia completely
solved the folklore case $\kappa=1$
of Conjecture \ref{GeneralConj}
\cite{BG22}
(that case of the conjecture 
seems to have been 
already conjectured
by Tian and Mazzeo in the 90's
and then explicitly stated
around 2009 by Donaldson
\cite{Mazzeo,JMR16,DonConic}), 
and interestingly
the Calabi ansatz makes an appearance in their proof as well.
It would be interesting to explore
whether some of the rather
elementary ideas here 
can be combined with some of
their deep estimates to
attack the general case
of Conjecture \ref{GeneralConj}.
\end{rmk}

\begin{rmk}

It is interesting to note that Abreu \cite[\S5]{Abreu} 
and subsequently Atiyah--LeBrun \cite[(5.5)--(5.6)]{A-LeBrun} also
studied a family of Einstein metrics with orbifold/cone-edge singularities 
converging (roughly) to the Eguchi--Hanson
metric. There are, however, important differences between
their construction and ours: (i) their metrics are {\it not } K\"ahler,
(ii) they work on the space $\PP^2$ while we work on the second
Hirzebruch surface $\FF_2$ (which generalizes to $\FF_{n,k}$ in
higher-dimensions),
(iii) the metrics they construct in the sequence
are singular along a single $\PP^1$ divisor instead of two 
$\PP^1$ divisors, as in our construction,
(iv) their limit metric is not precisely the Eguchi--Hanson metric,
but rather its double cover (they obtain an angle of $4\pi$ along
a $\PP^1$ divisor). 
While our metrics are constructed directly 
from the Calabi ansatz as K\"ahler--Einstein edge metrics,
Abreu constructs extremal (non-Einstein) K\"ahler metrics
using toric formalism and then uses Derdzinki's theorem to conclude
that a conformal rescaling (dividing by the scalar curvature squared)
will be Einstein (but not K\"ahler).
LeBrun obtains the same family of metrics using a completely different
method specific to dimension 4.  
One could take the double cover of our construction.
That would have the effect of, roughly speaking,
the curve being pushed out to infinity becoming smooth in the limit 
as the angle there would tend to $2\pi$ (instead of $\pi$),
while the other angle tends to $4\pi$. But even then it is not clear how to relate 
our metrics that are K\"ahler--Einstein edge (and in particular K\"ahler) to 
the Abreu--Atiyah--LeBrun metrics, that are not K\"ahler.
\end{rmk}
Motivated by the previous remark, we pose:

\begin{prob}
Relate the Abreu--Atiyah--LeBrun metrics to ours.
\end{prob}

A common feature for footballs,
Hirzebruch surfaces, and Calabi--Hirzebruch manifolds is
that there are two smooth disjoint divisors and hence two angle parameters $\beta_1,\beta_2$.
In Section \ref{sec: model}, we review the construction of Calabi--Hirzebruch manifolds \cite{Calabi82, H51}, denoted by $\mathbb{F}_{n, k}$ for $\NN\ni n\geq 2$ and $k\in\mathbb{N}$, and define two (families of) model metrics. The first
are Ricci-flat edge metrics on the total space of the (non-compact) line bundle $-kH_{\mathbb{P}^{n-1}}$,
\begin{equation*}
\omega_{\mathrm{eh}, n, k},
\end{equation*}
and an edge singularity of the angle $2\pi n/k$ along $Z_{n, k}$.
The second are 
compact \KEE spaces with positive Ricci curvature $(n+1)/k$ and an edge singularity of angle $2\pi/k$,
\begin{equation*}
\omega_{\mathrm{orb}, n, k}, 
\end{equation*}
on the weighted projective space $\mathbb{P}^n(1,\dots, 1,k)$.
These two model spaces turn out to be the
different Gromov--Hausdorff limits
of large-angle 
limits of the KEE metrics we construct 
on the Calabi--Hirzebruch
manifolds.

\subsection{Large angle limits}

The following result describes precisely
the different large-angle limits that arise from these KEE
metrics by either using different pointed limits or else 
parametrizing the angles in different ways.
Note that $\beta_1$ ranges in $(0,n/k)$ and 
$\beta_2$ ranges in $(0,1/k)$.

\begin{thm}\label{thm: main}
Fix a base point $p$ on the zero section of $\mathbb{F}_{n, k}$ and $q$ on the infinity section. The pointed metric space $(\mathbb{F}_{n, k}, \eta_{\beta_1}, p)$ converges in the pointed Gromov--Hausdorff sense to $(-kH_{\mathbb{P}^{n-1}}, \omega_{\mathrm{eh}, n, k}, p)$ as $\beta_1$ tends to $n/k$. On the other hand, $(\mathbb{F}_{n, k}, \xi_{\beta_2}, q)$ converges in the pointed Gromov--Hausdorff sense to $(\mathbb{P}^n(1,\dots, 1, k), \omega_{\mathrm{orb}, n, k}, q)$ as $\beta_2\to 1/k$. 

\end{thm}


In particular, when $n=k=2$, Theorem \ref{thm: main}
resolves
Conjecture \ref{RZConj}.
Theorem \ref{thm: main}
amounts to saying that the famous Eguchi--Hanson metric from mathematical physics is the Gromov--Hausdorff limit of {\it compact} \KE edge metrics that we construct on the second Hirzebruch surface (see Remark \ref{rmk: eh as limit} for more details or see Appendix \ref{app: eh another discussion} for another proof). 
More generally, when $n=k$, Theorem \ref{thm: main} recovers a family of Ricci-flat metrics on the total space of canonical bundle of $\mathbb{P}^{n-1}$ that was constructed by Calabi \cite{Calabi79},
once again as a limit of 
{\it compact} Einstein spaces.
\begin{figure}[!htbp]
    \centering
    \includegraphics[width=9.3cm]{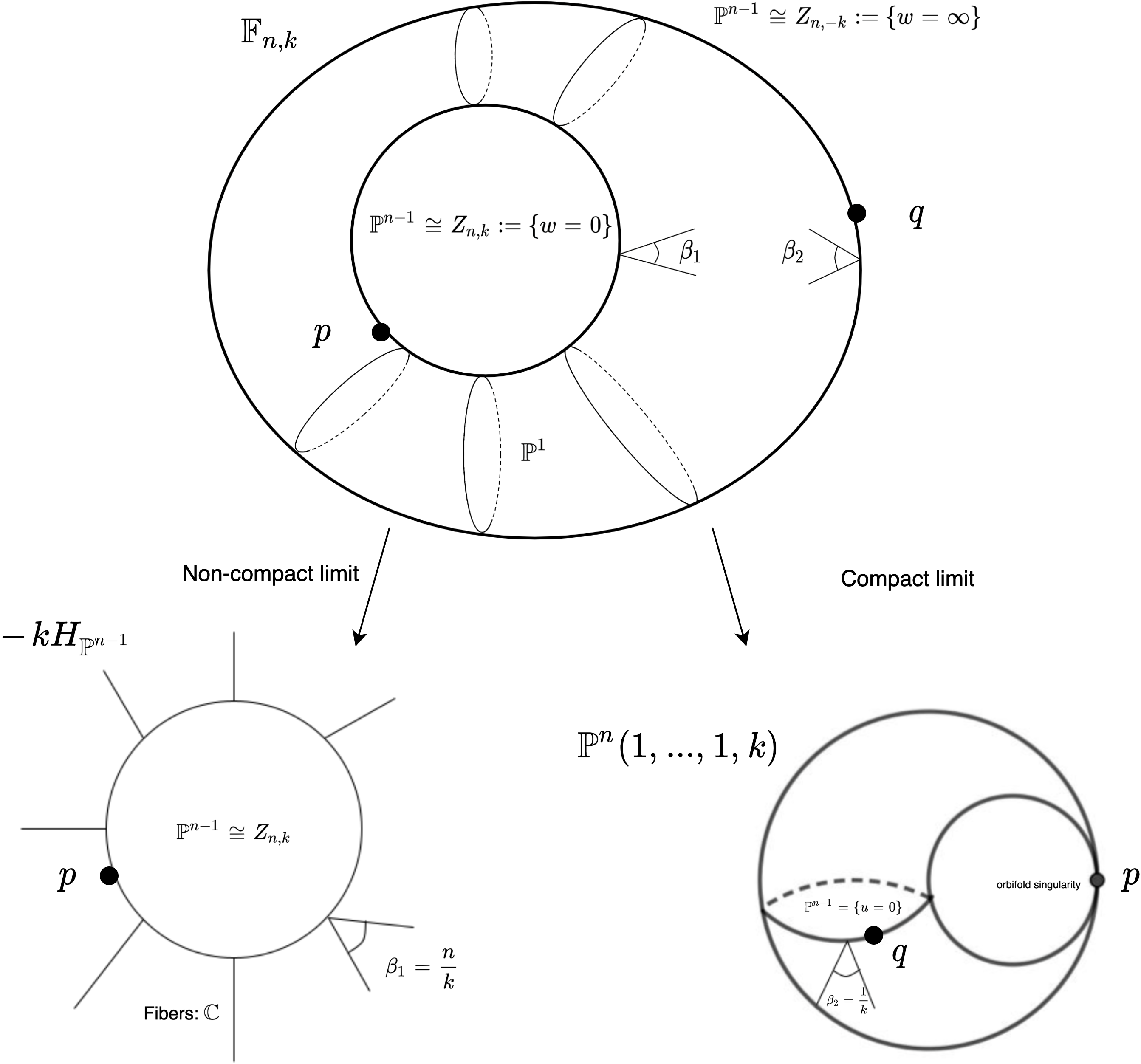}
    \caption{The upper part shows the K\"ahler edge structure on the Calabi--Hirzebruch manifold $\mathbb{F}_{n, k}$. When $n=2$, $\mathbb{F}_{2,k}$ is the $k$-th
    Hirzebruch surface and $Z_{2,\pm k}$ is a $\mp k$-curve. The lower part shows the different limits described in Theorem \ref{thm: main}. Note that on the left, we get a non-compact limit, with $q$ (together with all of $Z_{n,-k}$) pushed-out to infinity. On the right we get a compact limit, with $p$ limiting (together with all of $Z_{n,k}$) to an isolated orbifold point.}
    \label{fig:my_label}
\end{figure}
\begin{rmk}
It would be 
interesting to find physical
interpretations of Theorem 
\ref{thm: main} in the case 
$n=k=2$ of the Eguchi--Hanson
metric.
\end{rmk}

The elementary proof of Theorem \ref{thm: main} is divided into two parts. In section \ref{sec: asym of kee}, we study asymptotic behaviors of \KE edge metrics $\eta_{\beta_1}$ and $\xi_{\beta_2}$ when $\beta_1$, or respectively $\beta_2$, is close to $n/k$ or $1/k$. In section \ref{sec: gh kee}, we prove the convergence results by studying a family of ODEs that arises from the construction of $\eta_{\beta_1}$ and $\xi_{\beta_2}$.

\begin{rmk}\label{rmk_twocases}
The two families of \KE edge metrics $\eta_{\beta_1}$ and $\xi_{\beta_2}$ are
related via a simple, but important, rescaling. In Theorem \ref{thm: main}, we see
that different limits airse for those two family of metrics. The normalization factor can be obtained by studying the asymptotic behavior of $\xi_{\beta_2}$: see Proposition \ref{prop: length finite nkb2 case} for details.
\end{rmk}

\begin{rmk}
As noted before Theorem \ref{thm: main}, $\beta_1$ ranges in $(0,n/k)$ and 
$\beta_2$ ranges in $(0,1/k)$.
In proving Theorem \ref{thm: main}
for the family $\eta_{\beta_1}$
one obtains the asymptotic 
dependence of $\beta_2$ on $\beta_1$ in terms of a parameter
that controls the length of the 
fibers in the Hirzebruch fibration---see \eqref{Tbeta1beta2Eq}. This
shows that as $\beta_1$ tends
to its maximal value $n/k$, 
$\beta_2$ tends to its maximal
value as well, $1/k$, and 
the divisor $Z_{n,k}$ gets
pushed-off to infinity. 
It thus comes a bit as a surprise
that when we consider the family
$\xi_{\beta_2}$ parametrized in terms of $\beta_2$, and we let
$\beta_2$ tend towards $1/k$, 
while the parameter $\beta_1$ still tends towards its maximal
value $n/k$ we get a completely
different limiting behavior:
instead of a non-compact limit we get a compact limit via
a metric degeneration along $Z_{n, k}$.
Thus, studying the two families
is an essential feature of the setting and leads to two completely different Gromov--Hausdorff limits as in Theorem \ref{thm: main}.
It would be interesting to 
generalize this phenomenon to
other settings.
\end{rmk}

\subsection{Small angle limits}

Next, we resolve Conjecture \ref{GeneralConj} 
in the setting of $\mathbb{F}_{n, k}$. 
Denote by  $\widetilde{\eta_{\beta_1}}$ and
$\widetilde{\xi_{\beta_2}}$
fiber-wise rescalings of $\eta_{\beta_1}, \xi_{\beta_2},$ respectively
(see \eqref{eq: fibrescal} and \S\ref{Sec7}).

\begin{thm}\label{thm: main relation}
Both $(\mathbb{F}_{n, k}, \eta_{\beta_1})$ and 
$(\mathbb{F}_{n, k}, \xi_{\beta_2})$ converge in the Gromov--Hausdorff sense to 
$(\mathbb{P}^{n-1}, k\omega_{\operatorname{FS}})$ as $\beta_1$ or $\beta_2$ 
tends to $0$. Moreover,
as $\beta_1\searrow 0$, $(\mathbb{F}_{n, k}, \widetilde{\eta_{\beta_1}}, p)$ converges in the pointed Gromov--Hausdorff sense to $(\mathbb{P}^{n-1}\times \mathbb{C}^*, \frac{k}{n}(n\pi_1^*\omega_{\operatorname{FS}}+\pi_2^*\omega_{\operatorname{Cyl}}), p)$
and similarly for $(\mathbb{F}_{n, k}, \widetilde{\xi_{\beta_2}}, q)$,
where $p,q\in \mathbb{F}_{n, k}$ are as in Theorem \ref{thm: main}.

\end{thm}

A compendium of the limit theorems in \cite{RZ20, RZ21} and the 
present work is shown in Table 1.
{
\begin{table}[htbp]
\hspace{-2em}
\resizebox{1.1\textwidth}{!}{
\begin{tabular}{|c|c|c|c|c|c|c|}
\hline
$n$ & $k$ & $\beta_1$ & $\beta_2$ & Sequence & Limit & Reference \\
\hline\hline
$1$ & $k$ & $\beta_1\searrow 0$ & $\beta_2(\beta_1)\equiv\beta_1$ & football metrics &$(\mathbb{C}^*, \omega_{\mathrm{Cyl}})$ & folklore, \cite[Thm 1.3]{RZ20}\\
\hline
$1$ & $k$ & $\beta_1\nearrow 1$ & $\beta_2(\beta_1)\equiv\beta_1$ & football metrics & $\left(\mathbb{P}^1, \omega_{\operatorname{FS}}\right)$ & folklore, \cite[Thm 1.2]{RZ20}\\
\hline 
$2$ & $k$ & $\beta_1\searrow 0$ & $\beta_2(\beta_1)=\beta_1+O(\beta_1^2)$ & KEE metrics & $\left(\mathbb{P}^1, k\omega_{\operatorname{FS}}\right)$ & \cite[Thm 1.2]{RZ21}\\
\hline
$2$ & $k$ & $\beta_1\searrow 0$ & $\beta_2(\beta_1)=\beta_1+O(\beta_1^2)$ & rescaled KEE metrics & $\left(\mathbb{P}^1\times \mathbb{C}^*, k\pi_1^*\omega_{\operatorname{FS}}+k\pi_2^*\omega_{\operatorname{Cyl}}\right)$ & \cite[Thm 1.2]{RZ21}\\
\hline
$2$ & $1$ & $\beta_1\to 1$ & $\beta_2(\beta_1)\to \sqrt{3}-1$ & KEE metrics & KEE metric on $(\mathbb{F}_1, Z_{-1})$ & \cite[Cor 1.5]{RZ21}\\
\hline 
$2$ & $1$ & $\beta_1\nearrow 2$ & $\beta_2(\beta_1)\nearrow 1$ & KEE metrics & Ricci-flat metric on $(-H_{\mathbb{P}^1}, Z_1)$ & Thm \ref{thm: limit n2k1 case}\\
\hline 
$2$ & $1$ & the metric degenerates on $Z_1$ & $\beta_2\nearrow 1$ & KEE metrics & $(\mathbb{P}^2, \omega_{\operatorname{FS}})$& Thm \ref{thm: appthm2}.\\
\hline
$2$ & $1$ & $\beta_1(\beta_2)\nearrow 2$ & $\beta_2\nearrow 1$ & rescaled KEE metrics & Ricci-flat metric on $(-H_{\mathbb{P}^1}, Z_1)$ & Thm \ref{thm: rescale simcase}\\
\hline
$2$ & $2$ & $\beta_1\nearrow 1$ & $\beta_2(\beta_1)\nearrow 2$ & KEE metrics & Eguchi--Hanson metric ($\epsilon=1$) & Thm \ref{thm: metric limit nk case}\\
\hline
$2$ & $k$ & the metric degenerates on $Z_{2, k}$ & $\beta_2(\beta_1)\nearrow 1/k$ & KEE metrics & $(\mathbb{P}^2(1,1,k), \omega_{\mathrm{orb},2,k})$ & Thm \ref{thm: conv mod2}\\
\hline
$n$ & $1$ & the metric degenerates on $Z_{n, 1}$ & $\beta_2(\beta_1)\nearrow 1$ & KEE metrics & $(\mathbb{P}^n, \omega_{\mathrm{FS}})$ & Thm \ref{thm: conv mod2}\\
\hline
$n$ & $k$ & $\beta_1\nearrow n/k$ & $\beta_2(\beta_1)\nearrow 1/k$ & KEE metrics & $(-kH_{\mathbb{P}^{n-1}}, \omega_{\mathrm{eh}, n, k})$ & Thm \ref{thm: metric limit nk case}\\
\hline 
$n$ & $k$ & $\beta_1\searrow 0$ & $\beta_2(\beta_1)=\beta_1+O(\beta_1^2)$ & rescaled KEE metrics & $(\mathbb{P}^{n-1}\times \mathbb{C}^*, \frac{k}{n}(n\pi_1^*\omega_{\operatorname{FS}}+\pi_2^*\omega_{\operatorname{Cyl}}))$ & Thm \ref{thm: ResB1Case}\\
\hline
$n$ & $k$ & the metric degenerates on $Z_{n,k}$ & $\beta_2\nearrow 1/k$ & KEE metrics & $(\mathbb{P}^{n}(1,\dots,1,k), \omega_{\mathrm{orb}, n, k})$ & Thm \ref{thm: conv mod2}\\
\hline 
$n$ & $k$ & $\beta_1(\beta_2)\nearrow n/k$ & $\beta_2\nearrow 1/k$ & rescaled KEE metrics & $(-kH_{\mathbb{P}^{n-1}}, \omega_{\mathrm{eh}, n, k})$ & Cor \ref{thm: ResB2}\\
\hline 
$n$ & $k$ & $\beta_1(\beta_2)=\beta_2+O(\beta_2^2)$ & $\beta_2\searrow 0$ & rescaled KEE metrics & $(\mathbb{P}^{n-1}\times \mathbb{C}^*, \frac{k}{n}(n\pi_1^*\omega_{\operatorname{FS}}+\pi_2^*\omega_{\operatorname{Cyl}}))$ & Thm \ref{thm: ResB2Zero}\\
\hline
\end{tabular}
}
\caption{Limits of \KEE metrics on $\mathbb{F}_{n, k}$, the $k$-th Calabi--Hirzebruch manifold of dimension $n$.}
\label{summary}
\end{table}
}
\subsection{Organization}
In Section \ref{sec: model}, we first review the construction of Calabi--Hirzebruch manifolds and then define two families of model metrics respectively on the total space of line bundles $-kH_{\mathbb{P}^{n-1}}$ and the Calabi--Hirzebruch manifolds. In Section \ref{sec: general}, we construct \KE edge metrics on the Calabi--Hirzebruch manifolds following Calabi ansatz. In Section \ref{sec: asym of kee}, we study the asymptotic behaviors of \KE edge metrics by reducing it to the study of some ODEs. In Section \ref{sec: gh kee}, we consider the Gromov--Hausdorff limits of the \KE edge metrics on the Calabi--Hirzebruch manifolds and find out the model metrics in the limit.  Theorem \ref{thm: main} is then proved in two parts: Theorem \ref{thm: metric limit nk case} and Theorem \ref{thm: conv mod2}. 
The first statement of Theorem \ref{thm: main relation} about the non-rescaled limits
is proved in Theorems \ref{thm: b1 0 case} and \ref{thm: b2 0 case}.
The second statement of Theorem \ref{thm: main relation} concerning 
rescaled limits and pointed limits is contained in Theorems \ref{thm: ResB1Case}
and \ref{thm: ResB2Zero}. We also discuss the relation between Theorem \ref{thm: metric limit nk case} and Theorem \ref{thm: conv mod2} as mentioned in Remark \ref{rmk_twocases} in Corollary \ref{thm: ResB2}. 

Finally, we end with several appendices whose main purpose is to bridge the mathematical,
or ``Calabi's" formalism, and the physics, or ``Eguchi--Hanson" setting. These 
should make the article accessible to the physics audience, as well as introduce
some of the physics motivation for the article to the mathematical audience. 
Appendices \ref{app: review EH} and \ref{app: eh another discussion} provide a brief review on basic properties of Eguchi--Hanson metrics and explain how to understand Eguchi--Hanson metrics as Gromov--Hausdorff limits of \KE edge metrics. 
In Appendix \ref{app: more eg}, we provide more examples of \KE edge metrics and their Gromov--Hausdorff limit metrics. 

\paragraph{Acknowledgments.} 

Research supported by NSF grant 1906370, NSFC grant 12101052, the Fundamental Research Funds 2021NTST10 for the Central Universities, and a Brin Graduate
Fellowship at the University of
Maryland.
YAR thanks C. LeBrun for informing him of \cite{Abreu,A-LeBrun}.
The authors are grateful to the referees for their careful reading and comments.

This article is dedicated to Scott Wolpert that aside from
his seminal mathematical contributions to complex geometry
has been instrumental in forging the University of Maryland
as a leader in that field, and has positively impacted the 
careers of all three authors
as either graduate students
or faculty in the department
(and, in particular, in all likelihood this article
would not have appeared
had it not been for his constant support).

\section{Model metrics}\label{sec: model}

We first review the construction of Calabi--Hirzebruch manifolds \cite{Calabi82, H51}.

\begin{defi}\label{defi: CH manifolds}
The Calabi--Hirzebruch manifold, denoted by $\mathbb{F}_{n, k}$, where $n, k\in\mathbb{N}$, is defined as follows:
\begin{equation*}
    \mathbb{F}_{n, k}:=\mathbb{P}(-kH_{\mathbb{P}^{n-1}}\oplus\mathbb{C}_{\mathbb{P}^{n-1}}),\quad n, k\in\mathbb{N},
\end{equation*}
where $H_{\mathbb{P}^{n-1}}$ is the hyperplane bundle over $\mathbb{P}^{n-1}$ and $\mathbb{C}_{\mathbb{P}^{n-1}}$ is the trivial one. 

An alternative way to define $\mathbb{F}_{n, k}$ is by adding an infinity section to the blow up of $\mathbb{C}^n/\mathbb{Z}_k$ at the origin \cite[Lemma 2.1]{RZ21}.
\end{defi}

In this section, we introduce two model metrics $\omega_{\mathrm{eh}, n, k}$ and $\omega_{\mathrm{orb}, n, k}$, that are defined on the total space of line bundle $-kH_{\mathbb{P}^{n-1}}$ and 
on the weighted projective space
$\mathbb{P}^{n}(1,\dots,1,k)$, 
respectively. They will serve as the candidates of limit metrics in the latter sections. 

Denote by $[Z_1:\cdots:Z_n]$ the homogeneous coordinates on $\mathbb{P}^{n-1}$. Working on the chart $\{Z_i\neq 0\}$, we shall use the nonhomogeneous coordinates $z_j:=Z_j/Z_i$ for all $j\neq i$. Denote by $w$ the coordinate along each fiber on $-kH_{\mathbb{P}^{n-1}}$ or $\mathbb{F}_{n, k}$. Then $w\in\mathbb{C}$ for $-kH_{\mathbb{P}^{n-1}}$ and $w\in \mathbb{C}\cup \{\infty\}$ for $\mathbb{F}_{n, k}$. In particular, we have two divisors on $\mathbb{F}_{n, k}$: the zero section
\begin{equation*}
    Z_{n, k}:=\{w=0\}
\end{equation*}
and the infinity section
\begin{equation*}
    Z_{n, -k}:=\{w=\infty\}.
\end{equation*}
Consider the Hermitian norm on $-kH_{\mathbb{P}^{n-1}}$ (and also $\mathbb{F}_{n, k}$) 
\begin{equation*}
    ||(z_1,\dots, \hat{z}_i,\dots, z_n, w)||:=|w|^2\left(1+\sum_{j\neq i}^{n}|z_j|^2\right)^k,
\end{equation*}
on the chart $\{Z_i\neq 0\}$. Let $s$ be the logarithm of this fiberwise norm, i.e.,
\begin{equation}\label{eq: def s}
    s:=\log|w|^2+k\log\left(1+\sum_{j\neq i}|z_j|^2 \right).
\end{equation}
Next, we define model metrics that depend only on $s$ on $-kH_{\mathbb{P}^{n-1}}$ and $\mathbb{P}^n(1,\dots, 1, k)$.

\subsection{Non-compact case: model metrics on $-kH_{\mathbb{P}^{n-1}}$}
The first model metric, $\omega_{\mathrm{eh}, n, k}$, is a Ricci-flat edge metric with edge singularity of angle $2n\pi/k$ along the zero section $Z_{n, k}$ of $-kH_{\mathbb{P}^{n-1}}$. In particular, when $n=k$, this metric is smooth. Indeed, when $n=k$, $\omega_{\mathrm{eh}, n, k}$ coincides with the Ricci-flat metric on the canonical bundle of $\mathbb{P}^{n-1}$ that was constructed by Calabi \cite{Calabi79}.

\begin{defi}
Let $s$ be defined in \eqref{eq: def s}. To define $\omega_{\mathrm{eh}, n, k}$, introduce another coordinate $\lambda\in(1, +\infty)$ such that 
\begin{equation}\label{eq: lbd for c1}
    \lambda=(e^{\frac{n}{k}s}+1)^{\frac{1}{n}}.
\end{equation}
Thus $\{\lambda=1\}$ corresponds to the zero section $Z_{n, k}=\{w=0\}$ on $-kH_{\mathbb{P}^{n-1}}$. We define $\omega_{\mathrm{eh}, n, k}$ by giving its potential function as follows:
\begin{equation}\label{eq: poten c1}
    f(\lambda):=k\lambda +k\int\frac{1}{\lambda^n-1}\;\textrm{d}\lambda,
    \end{equation}
where the last term denotes an indefinite integral, and 
\begin{equation*}
\omega_{\mathrm{eh}, n, k}:=\sqrt{-1}\partial\bar{\partial}f(s).
\end{equation*}
The metric $\omega_{\mathrm{eh}, n, k}$ is a Ricci-flat metric on $-kH_{\mathbb{P}^{n-1}}$ with edge singularity of angle $2n\pi/k$ along the zero section (See Theorem \ref{thm: metric limit nk case} for details). It can be seen as a generalization of Eguchi--Hanson metrics to higher dimensional manifolds (See Remark \ref{rmk: eh spec case} for details). More precisely, in local coordinates, 
\begin{align*}
    \omega_{\mathrm{eh}, n, k}=k (e^{\frac{n}{k}s}+1)^{\frac{1}{n}} \pi_1^*\omega_{\operatorname{FS}}+\frac{e^{\frac{n}{k}s}}{k(e^{\frac{n}{k}s}+1)^{\frac{n-1}{n}}}&\left(\frac{\sqrt{-1}\textrm{d}w\wedge\textrm{d}\bar{w}}{|w|^2}+\sqrt{-1}\alpha\wedge\bar{\alpha}\right.+\\
    &\left.\sqrt{-1}\alpha\wedge\frac{\textrm{d}\bar{w}}{\bar{w}}+\sqrt{-1}\frac{\textrm{d}w}{w}\wedge\bar{\alpha}\right),
\end{align*}
where 
\begin{equation*}
    \alpha=k\frac{\sum_{i\neq j}\bar{z}_i\textrm{d}z_i}{1+\sum_{i\neq j}|z_i|^2},\quad \textup{on the chart}\;\{z_j\neq 0\},
\end{equation*}
and
\begin{equation*}
    \pi_1([Z_1:\cdots:Z_n], w)=[Z_1:\cdots:Z_n]
\end{equation*}
is the projection map from the total space $-kH_{\mathbb{P}^{n-1}}$ to the base space $\mathbb{P}^{n-1}$.
\end{defi}

\begin{rmk}\label{rmk: eh spec case}
Fixing $n=k=2$ in \eqref{eq: lbd for c1} and \eqref{eq: poten c1}, we obtain that on $-2H_{\mathbb{P}^1}$ the model metric $\omega_{\mathrm{eh}, n, k}$ has the expression
\begin{equation}\label{eq: mod sim eh}
    \sqrt{-1}\partial\bar{\partial}(2\lambda+\log(\lambda-1)-\log(\lambda+1)),\quad \lambda>1.
\end{equation}
Set
\begin{equation}\label{eq: def r for eh}
    r:=e^{\frac{1}{4}s},\quad r>0.
\end{equation}
Plugging \eqref{eq: lbd for c1} and \eqref{eq: def r for eh} in \eqref{eq: mod sim eh}, one finds $\omega_{\mathrm{eh}, n, k}$ has the following expression on $-2H_{\mathbb{P}^1}$:
\begin{equation}\label{eq: der eh}
    \sqrt{-1}\partial\bar{\partial}[\sqrt{r^4+1}+\log r^2-\log(\sqrt{r^4+1}+1)].
\end{equation}
\eqref{eq: der eh} coincides with \eqref{eq: eh metric potential form} in the Appendix where we put $\epsilon=1$. In other words, when $n=k=2$, $\omega_{\mathrm{eh}, n, k}$ recovers the famous Eguchi--Hanson metric \cite{EH79} that is constructed on the total space of $-2H_{\mathbb{P}^1}$.

It is recalled in Proposition \ref{prop: eh ricflat} that the Eguchi--Hanson metric is Ricci-flat. The metric $\omega_{\mathrm{eh}, n, k}$, which is also Ricci-flat but with edge singularities along $Z_{n, k}$ when $n\neq k$, can be seen as a generalization of Eguchi--Hanson metrics to the total space of $-kH_{\mathbb{P}^{n-1}}$ for arbitrary $n, k\in\mathbb{N}_{>0}$.
\end{rmk}

\subsection{Compact case: model metrics on $\mathbb{P}^{n}(1,\dots, 1,k)$}

In this section, we introduce another model metric $\omega_{\mathrm{orb}, n, k}$ that is defined on the weighted projective space
$\mathbb{P}^{n}(1,\dots,1,k)$. We first recall the construction of the weighted projective space and its orbifold structure.

\begin{defi}
For $k\in\mathbb{N}$, consider the group action of $\mathbb{C}^*$ on $\mathbb{C}^{n+1}\setminus\{0\}$ given by 
\begin{equation*}
    \lambda \cdot (z_0,\dots,z_{n-1},z_n) = (\lambda z_0,\dots, \lambda z_{n-1},\lambda^k z_n),\quad \lambda\in\mathbb{C}^*,\; (z_0,\dots, z_{n-1}, z_n) \in \mathbb{C}^{n+1}\setminus \{0\}.
\end{equation*}
Then the weighted projective space $\mathbb{P}^{n}(1,\dots,1,k)$ is defined as the quotient of this group action:
\begin{equation*}
    \mathbb{P}^{n}(1,\dots,1,k) := (\mathbb{C}^{n+1}\setminus\{0\})/\mathbb{C}^*.
\end{equation*}
We use homogeneous coordinates on $\mathbb{P}^n(1,\dots, 1,k)$. A point $[x_0:\cdots:x_{n-1}:x_n]\in\mathbb{P}^n(1,\dots,1,k)$ corresponds to an equivalence class in $\mathbb{C}^{n+1}\setminus \{0\}$. More precisely, 
\begin{equation}\label{eq_PonWps}
    [x_0:\cdots:x_{n-1}:x_n] := \{\lambda x_0,\dots, \lambda x_{n-1}, \lambda^k x_{n}: \lambda\in\mathbb{C}^*\}.
\end{equation}
\end{defi}

Consider the $\mathbb{Z}_k$ action on $\mathbb{C}^n$ such that the $\mathbb{Z}_k$ orbit of a point $(z_1,\dots, z_n)\in\mathbb{C}^n$ is
\begin{equation*}
    \{e^{\frac{2\pi\sqrt{-1}\ell}{k}}z_1,\dots, e^{\frac{2\pi\sqrt{-1}\ell}{k}}z_n: \ell = 0,\dots, k-1\}.
\end{equation*}

Then near the point $[0:0\cdots:0:1]\in\mathbb{P}^n(1,\dots, 1,k)$ the local structure is $\mathbb{C}^n/\mathbb{Z}_k$. The next lemma shows the relation between $\mathbb{P}^n(1,\dots,1,k)$ and $\mathbb{F}_{n, k}$.

\begin{lem}\label{lem_WpsAndCh}
The total space of the line bundle $kH_{\mathbb{P}^{n-1}}$ can be embedded in the weighted projective space $\mathbb{P}^{n}(1,\dots,1,k)$ and the complement is a single point $p=[0:\cdots:0:1]\in\mathbb{P}^n(1,\dots,1,k)$. Moreover, $\mathbb{F}_{n, k}$ is the blow up of $\mathbb{P}^{n}(1,\dots,1,k)$ at $p$.
\end{lem}

\begin{proof}
Since for any vector bundle $A$ and line bundle $L$ we have $\mathbb{P}(A\otimes L)=\mathbb{P}(A)$,  it follows
that $\mathbb{F}_{n, -k}$ is biholomorphic to $\mathbb{F}_{n, k}$ by taking $L=2kH_{\mathbb{P}^{n-1}}$ in Definition \ref{defi: CH manifolds} with the biholomorphism exchanging the
zero and the infinity sections $Z_{n, k}$ and $Z_{n, -k}$. Thus, we identify $\mathbb{F}_{n, k}$ as 
\begin{equation}\label{eq_ExCh}
    \mathbb{F}_{n, k} = \mathbb{P}(kH_{\mathbb{P}^{n-1}}\oplus \mathbb{C}_{\mathbb{P}^{n-1}}).
\end{equation}
We first show the total space of $kH_{\mathbb{P}^{n-1}}$ can be naturally embedded in $\mathbb{P}^n(1,\dots,1,k)$.

Consider $\mathbb{P}^{n}(1,\dots,1,k)$ with homogeneous coordinate $[x_0:\cdots, x_n]$ and embed $\mathbb{P}^{n-1}$ in $\mathbb{P}^n$ as $\{x_n=0\}$. Recall the transition function at a point $[x_0:\cdots:x_{n-1}]\in\mathbb{P}^{n-1}$ of the line bundle $kH_{\mathbb{P}^{n-1}}$ is given by 
\begin{equation}\label{eq_TranFunc}
    g_{ij}([x_0:\cdots:x_{n-1}]) = \left(\frac{x_j}{x_i}\right)^k.
\end{equation}
Then the fiber of $kH_{\mathbb{P}^{n-1}}$ at an arbitrary point $[x_0:\cdots,x_{n-1}:0]\in\mathbb{P}^{n}(1,\dots,1,k)$ can be identified as the set of all $[x_0:\cdots:x_{n-1}:\lambda]$ for $\lambda\in\mathbb{C}$, where $[x_0:\cdots:x_{n-1}:\lambda]$ is defined in \eqref{eq_PonWps}. By definition of the weighted projective space and \eqref{eq_TranFunc}, this is well-defined. Thus, $kH_{\mathbb{P}^{n-1}}$ can be naturally embedded in $\mathbb{P}^n(1,\dots, 1,k)$ and the complement is the point $[0:\cdots:0:1]$. We have mentioned that the local structure of $\mathbb{P}^n(1,\dots, 1, k)$ near this point is $\mathbb{C}^n/\mathbb{Z}_k$. 

Next, we blow up $\mathbb{P}^n(1,\dots,1,k)$ at $[0:\cdots:0:1]$. The exceptional divisor can be identified as adding an infinity section that is biholomorphic to $\mathbb{P}^{n-1}$ to the total space $kH_{\mathbb{P}^{n-1}}$. Combining this observation with \eqref{eq_ExCh}, one realizes the manifold upstairs is the Calabi--Hirzebruch manifold $\mathbb{F}_{n, k}$. We henceforth denote by $\pi$ the blow down map.
\end{proof}


Now we are ready to define $\omega_{\mathrm{orb},n,k}$. To do so, we first define another family of model metrics $\tilde{\omega}_{\mathrm{orb}, n, k}$ on $\mathbb{F}_{n, k}$. The metric $\tilde{\omega}_{\mathrm{orb}, n, k}$ is a \KE edge metric with Ricci curvature $(n+1)/k$ and an edge singularity of angle $2\pi/k$ along $Z_{n, -k}$ that degenerates on $Z_{n, k}$. In particular, this metric is smooth along $Z_{n, -k}$ when $k=1$ and collapses along $Z_{n, k}$. Indeed, when $k=1$ this metric coincides with the Fubini--Study metric on $\mathbb{P}^n$. We will then define $\omega_{\mathrm{orb},n,k}$ as the pull-back of the metric $\tilde{\omega}_{\mathrm{orb},n,k}$ under the blow up map.

\begin{defi}\label{def_ModOnFnk}
Let $s$ be defined in \eqref{eq: def s}. To define $\tilde{\omega}_{\mathrm{orb}, n, k}$, introduce another coordinate $\nu\in(0, 1)$ such that 
\begin{equation*}
    \nu=1-(e^{\frac{s}{k}}+1)^{-1}.
\end{equation*}
Note that $\{\nu=0\}$ and $\{\nu=1\}$ correspond to the zero section $Z_{n, k}$ and the infinity section $Z_{n, -k}$ respectively. We define $\tilde{\omega}_{\mathrm{orb}, n, k}$ by giving its potential function as follows:
\begin{equation*}
    g(\nu):=-k\log(1-\nu), \quad \nu\in(0, 1),
    \end{equation*}
and
\begin{equation*}
\tilde{\omega}_{\mathrm{orb}, n, k}:=\sqrt{-1}\partial\bar{\partial}g(s).
\end{equation*}
Note that $\tilde{\omega}_{\mathrm{orb}, n, k}$ is a degenerate \KE edge metric on $\mathbb{F}_{n, k}$ with Ricci curvature $(n+1)/k$ and an edge singularity of the angle $2\pi/k$ along $Z_{n, -k}$, while $\tilde{\omega}_{\mathrm{orb}, n, k}$ collapses along $Z_{n, k}$ (See Theorem \ref{thm: conv mod2} for details). More precisely, in local coordinates, 
\begin{align*}
    \tilde{\omega}_{\mathrm{orb}, n, k}=\frac{ke^{\frac{s}{k}}}{e^{\frac{s}{k}}+1}\pi_1^*\omega_{\operatorname{FS}}+\frac{e^{\frac{s}{k}}}{k(e^{\frac{s}{k}}+1)}&\left(\frac{\sqrt{-1}\textrm{d}w\wedge\textrm{d}\bar{w}}{|w|^2}+\sqrt{-1}\alpha\wedge\bar{\alpha}+\right.\\
    &\left.\sqrt{-1}\alpha\wedge\frac{\textrm{d}\bar{w}}{\bar{w}}+\sqrt{-1}\frac{\textrm{d}w}{w}\wedge\bar{\alpha}\right),
\end{align*}
where 
\begin{equation*}
    \alpha=k\frac{\sum_{i\neq j}\bar{z}_i\textrm{d}z_i}{1+\sum_{i\neq j}|z_i|^2},\quad \textup{on the chart}\;\{z_j\neq 0\},
\end{equation*}
and
\begin{equation*}
    \pi_1([Z_1:\cdots:Z_n], w)=[Z_1:\cdots:Z_n]
\end{equation*}
is the projection map from $\mathbb{F}_{n, k}$ to the zero section $Z_{n, k}$ identified as $\mathbb{P}^{n-1}$.
\end{defi}

\begin{defi}\label{def_ModOnWps}
For $k>1$, the model metric $\omega_{\mathrm{orb},n,k}$ on $\mathbb{P}^n(1,\dots,1,k)$ is defined away from $[0:\cdots:0:1]$ as the pull-back of the metric $\tilde{\omega}_{\mathrm{orb},n,k}$ on $\mathbb{F}_{n, k}\setminus\{Z_{n, k}\}$ under the inverse of the blow up map. For $k=1$, we define $\omega_{\mathrm{orb},n,1}$ as the Fubini--Study metric on $\mathbb{P}^n(1,\dots,1,1)=\mathbb{P}^n$.
\end{defi}



\section{Constructions of K\"{a}hler--Einstein edge metrics }\label{sec: general}

In this section, we aim to construct \KE edge metrics defined on the Calabi--Hirzebruch manifold $\mathbb{F}_{n, k}$ with edge singularities along the zero section $Z_{n, k}$ and the infinity section $Z_{n, -k}$ using Calabi ansatz. By \KE edge metrics we mean K\"{a}hler edge currents that satisfy the \KE equation on smooth locus. The reader may refer to \cite{Y14} for detailed exposition of K\"{a}hler edge metrics.

\subsection{Constructions of $\eta_{\beta_1}$ on $\mathbb{F}_{n, k}$}

Recall, the coordinate $s$ is defined in \eqref{eq: def s}, where $w$ and $z_j$, $j\neq i$ are coordinates we use when working on the chart $\{Z_i\neq 0\}\subset \mathbb{F}_{n, k}$. We seek a K\"{a}hler--Einstein edge metric 
\begin{equation*}
  \eta:=\sqrt{-1}\partial\bar{\partial}f(s)
\end{equation*}
for some smooth function $f$ on $\mathbb{F}_{n, k}$. Define
\begin{align}
\begin{aligned}\label{eq: defi of tau and varphi nk case}
    \tau(s):&=f'(s),\quad s\in(-\infty, +\infty)\\
    \varphi(s):&=f''(s)=\tau'(s),\quad s\in(-\infty, +\infty).
    \end{aligned}
\end{align}

Let $\pi_1$ and $\pi_2$ be projections from $\mathbb{F}_{n, k}$ to the zero section $Z_{n, k}$ and each fiber respectively, i.e., 
\begin{align*}
    \pi_1([Z_1:\cdots:Z_n], w)&=[Z_1:\cdots:Z_n],\\ \pi_2([Z_1:\cdots:Z_n], w)&=w.
\end{align*}
Denote by $\omega_{\operatorname{FS}}$ the Fubini--Study metric on $\mathbb{P}^{n-1}$ and define
\begin{align*}
    \omega_{\operatorname{Cyl}}&:=\frac{\sqrt{-1}\textrm{d}w\wedge\textrm{d}\bar{w}}{|w|^2},\\
    \alpha&:=\frac{k\sum_{j\neq i}\bar{z}_j\textrm{d}z_j}{1+\sum_{j\neq i}|z_j|^2},\quad \text{on the chart}\; \{Z_i\neq 0\}.
\end{align*}

Then direct calculation (cf. \cite[Section 3.3]{RZ21} for detailed computations in dimension $n$=2) yields
\begin{equation}\label{eq: eta in nk case}
    \begin{aligned}
    \eta=k\tau\pi_1^*\omega_{\operatorname{FS}}+\varphi&\left(\pi_2^*\omega_{\operatorname{Cyl}}+\sqrt{-1}\alpha\wedge\bar{\alpha}+\right.\\
    &\left.\sqrt{-1}\alpha\wedge\frac{\textrm{d}\bar{w}}{\bar{w}}+\sqrt{-1}\frac{\textrm{d}w}{w}\wedge\bar{\alpha}\right).
    \end{aligned}
\end{equation}

Positive definiteness of $\eta$ then implies $f'\geq 0$ and $f''\geq 0$ for $s\in(-\infty, +\infty)$.

\begin{prop}\label{prop: inf sup cond}
Assume $Z_{n, k}$ is non-collapsed, then $\inf_{s\in\mathbb{R}} \tau(s)>0$. Moreover, $\sup_{s\in\mathbb{R}} \tau(s)<+\infty$.
\end{prop}

\begin{proof}
Assume by contradiction that $\inf_{s\in\mathbb{R}} \tau(s)=0$. Then by \eqref{eq: eta in nk case}, $\eta$ is identically zero when restricted to $Z_{n, k}$. However, we assume $Z_{n, k}$ is non-collapsed. Thus we must have $\inf_{s\in\mathbb{R}} \tau(s)>0$. To see $\sup_{s\in\mathbb{R}}\tau(s)<+\infty$, we follow the arguments in \cite[Lemma 3.2]{RZ21}. Indeed, restricting $\eta$ to the fiber $\{Z_j = 0, \forall j\neq i\}$, in \eqref{eq: eta in nk case} we get 
\begin{equation}
\begin{aligned}\label{eq: restriction}
    \eta &= \varphi \cdot \pi_2^* \omega_{\operatorname{Cyl}}\\
    &= \varphi \frac{\sqrt{-1}\textrm{d}w\wedge\textrm{d}\bar{w}}{|w|^2},
\end{aligned}
\end{equation}
and also in this case $s = \log|w|^2$. Then we use the coordinate $w=e^{s/2+\sqrt{-1}\theta}$ on this fiber, and by \eqref{eq: restriction} $\eta$ restricted on the fiber gives a metric
\begin{equation}\label{eq: metric res}
    g_\eta = \frac{1}{2\varphi(\tau)} \textrm{d}\tau^2+2\varphi(\tau)\textrm{d}\theta^2.
\end{equation}
Thus, the volume form on the fiber is $\textrm{d}\tau\wedge \textrm{d}\theta$, and the volume of the fiber is $2\pi (\sup_{s\in\mathbb{R}}\tau(s)-\inf_{s\in\mathbb{R}}\tau(s))$. Since for any K\"{a}hler edge metric the volume of a complex submanifold is finite, we conclude $\sup_{s\in\mathbb{R}}\tau(s)<\infty$. 
\end{proof}
By Proposition \ref{prop: inf sup cond}, we may rescale $f$ by a positive constant such that $\inf \tau=1$ and $\sup \tau=T<\infty$. Thus, we henceforth assume that $\tau$ ranges from $[1, T]$.

\begin{prop}\label{prop: bd cond b1 case}
If $\eta$ is a K\"{a}hler edge metric with conic singularities along $Z_{n, k}$ and $Z_{n, -k}$, then
\begin{align}
\begin{aligned}\label{eq: bd condition nk case}
    \varphi(1)&=0,\quad \frac{\textrm{d}\varphi}{\textrm{d}\tau}(1)=\beta_1,\\
    \varphi(T)&=0,\quad \frac{\textrm{d}\varphi}{\textrm{d}\tau}(T)=-\beta_2,
    \end{aligned}
\end{align}
if we denote by $2\pi\beta_1$ and $2\pi\beta_2$ respectively the angles of $\eta$ along $Z_{n, k}$ and $Z_{n, -k}$.
\end{prop}
\begin{proof}
Recall we assume $\tau$ ranges from $[1, T]$. In particular, this implies that
\begin{equation*}
    \lim_{s\to \pm \infty}\frac{\textrm{d}\tau}{\textrm{d}s}=0,
\end{equation*}
which combines with \eqref{eq: defi of tau and varphi nk case} show that
\begin{equation*}
    \varphi(1)=\varphi(T)=0.
\end{equation*}
Next, we follow the arguments in the proof of \cite[Proposition 3.3]{RZ21}. Indeed, it follows from \cite[Theorem 1, proposition 4.4]{JMR16} that the potential function $f$ of $\eta$ has complete asymptotic expansions near both $w=0$ and $w=\infty$. Near $w=0$, the leading term in the expansion is $|w|^{2\beta_1}$. More precisely, by \eqref{eq: def s},
\begin{equation*}
    \begin{aligned}
        \varphi &\sim C_1 + C_2|w|^{2\beta_1}+(C_3\sin\theta+C_4\cos\theta)|w|^2+O(|w|^{2+\epsilon})\\
        &= C_1 + C_2 e^{\beta_1 s}+(C_3\sin\theta+C_4\cos\theta)e^s+O(e^{(1+\epsilon)s}).
    \end{aligned}
\end{equation*}
We first find $C_1=0$ by the fact that $\varphi(1)=\varphi(T)=0$ in \eqref{eq: bd condition nk case}. Moreover, the expansion can be differentiated term-by-term as $|w|\to 0$ or $s\to -\infty$. As $\varphi'(\tau)=\frac{\partial \varphi}{\partial s}\frac{\textrm{d}s}{\textrm{d}\tau}=\frac{\partial \varphi}{\partial s}/\varphi$, we obtain
\begin{equation*}
    \varphi'(1) = \beta_1.
\end{equation*}
The same arguments imply that
\begin{equation*}
    \varphi'(T) = -\beta_2,
\end{equation*}
where the minus sign comes from the fact that the leading term in this expansion is $|w|^{-2\beta_2}=e^{-\beta_2 s}$.
\end{proof}

\begin{defi}\label{defi: KEE}
A K\"{a}hler--Einstein edge metric $\omega$ is a K\"{a}hler edge metric that has constant Ricci curvature away from the singular locus.
\end{defi}

By Definition \ref{defi: KEE}, $\eta$ is a \KE edge metric if and only if it satisfies the following \KE edge equation:
\begin{equation}\label{eq: eta kee}
    \operatorname{Ric}\eta=\lambda\eta+(1-\beta_1)[Z_{n, k}]+(1-\beta_2)[Z_{n, -k}],
\end{equation}
where by $\lambda$ we denote the Ricci curvature.

The next proposition shows that \eqref{eq: eta kee} is equivalent to an ODE satisfied by $\tau$ and $\varphi$ with boundary conditions given in \eqref{eq: bd condition nk case}.

\begin{prop}\label{prop: red to ode}
The \K\; edge metric $\eta$ given in \eqref{eq: eta in nk case} satisfies the \KE edge equation \eqref{eq: eta kee} if and only if
\begin{align}
    &\lambda=\frac{n}{k}-\beta_1,\quad \beta_1\in\left(0, \frac{n}{k}\right), \label{eq: ric curv gene}\\
    &\varphi(\tau)=\frac{1}{k}\frac{\tau^n-1}{\tau^{n-1}}+\frac{1}{n+1}(\beta_1-\frac{n}{k})\frac{\tau^{n+1}-1}{\tau^{n-1}}, \quad \tau\geq 1.\label{eq: varphi general expression nk case}
\end{align}
Moreover, $\beta_2$ and $T$ are determined by $\beta_1$ such that $T>1$ and $\beta_2 > 0$.
\end{prop}
\begin{proof}
By \eqref{eq: eta in nk case}, we calculate $\operatorname{Ric}\eta$:
\begin{align}
\begin{aligned}\label{eq: ric eta nk case}
    \operatorname{Ric}\eta&=-\sqrt{-1}\partial\bar{\partial}\log \eta^n\\
    &=-\sqrt{-1}\partial\bar{\partial} \log \left(k^{n-1}\tau^{n-1}\varphi(\pi_1^* \omega_{\operatorname{FS}})^{n-1}\wedge\pi_2^*\omega_{\operatorname{Cyl}} \right)\\
    &=(1-\beta_1)[Z_{n, k}]+(1-\beta_2)[Z_{n, -k}]+\left(n-k(n-1)\frac{\varphi}{\tau}-k\frac{\textrm{d}\varphi}{\textrm{d}\tau}\right)\pi_1^*\omega_{\operatorname{FS}}\\
    &-\varphi \frac{\textrm{d}}{\textrm{d}\tau}\left((n-1)\frac{\varphi}{\tau}+\frac{\textrm{d}\varphi}{\textrm{d}\tau} \right)\left(\pi_2^*\omega_{\operatorname{Cyl}}+\sqrt{-1}\alpha\wedge\bar{\alpha}+\right.\\
    &\left.\sqrt{-1}\alpha\wedge\frac{\textrm{d}\bar{w}}{\bar{w}}+\sqrt{-1}\frac{\textrm{d}w}{w}\wedge\bar{\alpha} \right).
    \end{aligned}
\end{align}
Plugging \eqref{eq: eta in nk case} and \eqref{eq: ric eta nk case} in \eqref{eq: eta kee}, \eqref{eq: eta kee} is equivalent to 
\begin{align}
    n-k(n-1)\frac{\varphi}{\tau}-k\frac{\textrm{d}\varphi}{\textrm{d}\tau}&=\lambda k \tau,\label{eq: l1 ode}\\
    -\varphi\frac{\textrm{d}}{\textrm{d}\tau}((n-1)\frac{\varphi}{\tau}+\frac{\textrm{d}\varphi}{\textrm{d}\tau})&=\lambda \varphi.\label{eq: l2 ode}
\end{align}
Now we first observe that \eqref{eq: l1 ode} and \eqref{eq: l2 ode} are equivalent since taking derivative of \eqref{eq: l1 ode} with respect to $\tau$ gives \eqref{eq: l2 ode}. Thus we conclude \eqref{eq: eta kee} is equivalent to the ODE \eqref{eq: l1 ode} together with boundary conditions \eqref{eq: bd condition nk case}. Solving this ODE with boundary conditions gives \eqref{eq: ric curv gene} and \eqref{eq: varphi general expression nk case}. We need the assumption $\beta_1\in(0, n/k)$ to ensure the existence of $T$ and $\beta_2$ that satisfy \eqref{eq: bd condition nk case}. Indeed, $\beta_2$ and $T$ are determined by $\beta_1$ due to \eqref{eq: bd condition nk case} and \eqref{eq: varphi general expression nk case}. More precisely, $T$ should be a root of the polynomial that appears in the right hand side of $\eqref{eq: varphi general expression nk case}$ and $-\beta_2$ should be determined by the derivative $\textrm{d}\varphi/\textrm{d}\tau$ at $T$. Writing \eqref{eq: varphi general expression nk case} as $\varphi(\tau)=P(\tau)/\tau^{n-1}$, we factor $P(\tau)$ as
\begin{equation}\label{eq: polyforT}
\begin{aligned}
    P(\tau) &= (\tau-1)\left[\frac{1}{k}(\tau^{n-1}+\cdots+1)+\frac{1}{n+1}\left(\beta_1-\frac{n}{k}\right)(\tau^n+\cdots+1)
    \right]\\
    &= \frac{(\tau-1)}{n+1}\left[
    \left(\beta_1-\frac{n}{k}\right)\tau^n+\left(\frac{1}{k}+\beta_1\right)(\tau^{n-1}+\cdots +1)
    \right].
\end{aligned}
\end{equation}
Under the assumption $\beta_1\in(0, n/k)$, one finds $P(\tau)>0$ for $1<\tau\ll 2$ and $P(\tau)<0$ for $\tau\to \infty$. Thus, $P(\tau)$ has at least one real root that is greater than $1$. Denote by $T$ 
the first root of $P$ after $1$. In particular, $\textrm{d}\varphi/\textrm{d}\tau(T)<0$. By \eqref{eq: bd condition nk case}, this implies that $\beta_2>0$ as claimed.
\end{proof}

By \eqref{eq: defi of tau and varphi nk case}, there holds
\begin{equation}\label{eq: der sty}
    \frac{\textrm{d}s}{\textrm{d}\tau}=\frac{1}{\varphi(\tau)}.
\end{equation}
Combining \eqref{eq: eta in nk case}, \eqref{eq: der sty} and Proposition \ref{prop: red to ode} altogether, we realize that given $\beta_1\in(0, n/k)$ we can construct a \KE edge metric $\eta$ on $\mathbb{F}_{n, k}$ using coordinates $\tau$ and $\varphi$. This \KE edge metric has Ricci curvature $\lambda=n/k-\beta_1$, and has edge singularities of angle $2\pi\beta_1$ along $Z_{n, k}$ and angle $2\pi\beta_2$ along $Z_{n, -k}$. Note that $\beta_2$ is determined by $\beta_1$. In other words, we can construct a family of \KE edge metrics on $\mathbb{F}_{n, k}$ parametrized by $\beta_1\in (0, n/k)$. In Section \ref{sec: asym of kee}, we study the asymptotic behavior of this family of metrics when $\beta_1$ approaches the two extremes: $n/k$ or $0$.

Recall, in the construction of $\eta$ we rescale the metric such that $\tau$ ranges from $[1, T]$. Note that $T$ is determined by $\beta_1$. By Proposition \ref{prop: inf sup cond}, we may also rescale the metric such that $\tau$ ranges from $[t, 1]$ for some $0<t<1$. In such a way, we construct another family of \KE edge metrics on $\mathbb{F}_{n, k}$ parametrized by $\beta_2$ in the remainder of this section. Such metrics can be seen as obtained after renormalizing metrics in the family that is parametrized by $\beta_1$. However, when studying their asymptotic behaviors they give rise to different limit metric.

\subsection{Constructions of $\xi_{\beta_2}$ on $\mathbb{F}_{n, k}$}

Now consider a change of coordinate $u:=1/w$. Then \eqref{eq: def s} can be written as 
\begin{equation*}
    s=-\log|u|^2+k\log\left(1+\sum_{j\neq i}^n |z_j|^2 \right).
\end{equation*}
We still denote by $f(s)$ a smooth function on $\mathbb{F}_{n, k}$ and seek \KE edge metrics that have the form $\sqrt{-1}\partial\bar{\partial} f(s)$. Recall $\tau(s)$ and $\varphi(s)$ are defined in \eqref{eq: defi of tau and varphi nk case}. Let
\begin{align}
    \xi:&=\sqrt{-1}\partial\bar{\partial}f(s)\notag\\
    &\begin{aligned}
    &=k\tau\pi_1^*\omega_{\operatorname{FS}}+\varphi\left(\pi_2^*\omega_{\operatorname{Cyl}}+\sqrt{-1}\alpha\wedge\bar{\alpha}-\right.\\
    &\left.\sqrt{-1}\alpha\wedge\frac{\textrm{d}\bar{u}}{\bar{u}}-\sqrt{-1}\frac{\textrm{d}u}{u}\wedge\bar{\alpha} \right),\label{eq: xi metric}
    \end{aligned}
\end{align}
where $\pi_1$, $\pi_2$ and $\alpha$ are the same as those in \eqref{eq: eta in nk case}. Recall by Proposition \ref{prop: inf sup cond},
\begin{equation*}
0<\inf_{s\in\mathbb{R}}\tau(s)<\sup_{s\in\mathbb{R}}\tau(s)<+\infty.
\end{equation*}
We rescale the potential function $f$ such that for some $t>0$,
\begin{equation}\label{eq: new iod tau}
    \sup_{s\in\mathbb{R}} \tau(s)=1,\quad \inf_{s\in\mathbb{R}} \tau(s)=t.
\end{equation}
Assume $\xi$ is a \K\;edge metric on $\mathbb{F}_{n, k}$. Then, under the renormalization given by \eqref{eq: new iod tau}, Proposition \ref{prop: bd cond b1 case} and Proposition \ref{prop: red to ode}, which respectively describe the boundary conditions for $\varphi$ and the ODE satisfied by $\varphi$, translate to the following two Propositions.

\begin{prop}\label{prop: bd cond b2 case}
Assume $\xi$ has edge singularities of angle $2\pi\beta_1$ and $2\pi\beta_2$ respectively along $Z_{n, k}$ and $Z_{n, -k}$. Recall $\varphi$ in \eqref{eq: xi metric}. Then,
\begin{align}
\begin{aligned}\label{eq: bd cond nk b2 case}
    \varphi(t)&=0,\quad \frac{\textrm{d}\varphi}{\textrm{d}\tau}(t)=\beta_1,\\
    \varphi(1)&=0,\quad \frac{\textrm{d}\varphi}{\textrm{d}\tau}(1)=-\beta_2.
    \end{aligned}
\end{align}
\end{prop}

\begin{prop}\label{prop: red to ode b2 case}
Under the same assumptions in Proposition \ref{prop: bd cond b2 case}, the \K\;edge metric $\xi$ satisfies the \KE edge equation if and only if 
\begin{align}
    &\mu=\frac{n}{k}+\beta_2,\quad \beta_2\in\left(0, \frac{1}{k}\right),\label{eq: beta2range}\\
    &\varphi(\tau)=\frac{1}{k}\frac{\tau^n -1}{\tau^{n-1}}-\frac{1}{n+1}\left(\frac{n}{k}+\beta_2\right)\frac{\tau^{n+1}-1}{\tau^{n-1}},\quad \tau\in(t, 1),\notag
\end{align}
where by $\mu$ we denote the Ricci curvature of $\xi$. Moreover, $\beta_1$ and $t$ are determined by $\beta_2$.
\end{prop}

By Propositions \ref{prop: bd cond b2 case} and \ref{prop: red to ode b2 case}, we realize that given $\beta_2$ we can construct a family of \KE edge metrics on $\mathbb{F}_{n, k}$ parametrized by $\beta_2$ such that $\beta_1$ and $t$ are determined by $\beta_2$. This family of metrics can be obtained by renormalizing the family of metrics parametrized by $\beta_1$.

\section{Angle asymptotics 
}\label{sec: asym of kee}
In this section, we study the asymptotic behaviors of \KE edge metrics $\eta$ in \eqref{eq: eta in nk case} (respectively, $\xi$ in \eqref{eq: xi metric}) as $\beta_1$ (respectively, $\beta_2$) approaches either of its extremes. 

We first study the limit behavior of $\eta$ when $\beta_1\nearrow n/k$. It is enough to study the limiting behavior of $\varphi$ and $\tau$ thanks to Proposition \ref{prop: red to ode}.

\begin{prop}\label{prop: T infty nk case}
When $\beta_1$ is close to $n/k$, we have $T>1$ and $T\to \infty$ as $\beta_1\nearrow n/k$.
\end{prop}
\begin{proof}
Recall by \eqref{eq: bd condition nk case}, $1$ and $T$ are both roots of $\varphi(\tau)$. By Proposition \ref{prop: red to ode}, $T>1$. To prove $T\to\infty$ as $\beta_1\nearrow n/k$, we write \eqref{eq: varphi general expression nk case} as \begin{equation}\label{eq: varphi roots expre}
    \varphi(\tau)=\frac{1}{\tau^{n-1}}\cdot\frac{1}{n+1}(\beta_1-\frac{n}{k})(\tau-1)(\tau-\alpha_1)\cdots(\tau-\alpha_{n-1})(\tau-T),
\end{equation}
where $\alpha_1, \dots, \alpha_{n-1}\in\mathbb{C}$. Comparing \eqref{eq: varphi general expression nk case} to \eqref{eq: varphi roots expre}, we have
\begin{equation*}
    (\tau-\alpha_1)\cdots(\tau-\alpha_{n-1})(\tau-T)=\tau^n+\frac{1+k\beta_1}{k\beta_1-n}(\tau^{n-1}+\cdots+1).
\end{equation*}
By Vieta's formulas,
\begin{align}
    (-1)^n\alpha_1\cdots\alpha_{n-1}\cdot T&=\frac{1+k\beta_1}{k\beta_1-n},\label{eq: product}\\
    \alpha_1+\cdots+\alpha_{n-1}+T&=\frac{1+k\beta_1}{n-k\beta_1}.\label{eq: sum}
\end{align}


By \eqref{eq: varphi general expression nk case}, $\varphi(\tau)$ has $n+1$ (possibly complex) roots. Since
\begin{equation*}
    \varphi(\tau)\to \frac{1}{k}\frac{\tau^n-1}{\tau^{n-1}},\quad\;\textrm{as}\;\beta_1\nearrow n/k,
\end{equation*}
we conclude that the $n$ roots of $\varphi$ except $T$ converge to the $n$th root of unity as $\beta_1\nearrow n/k$. In particular, $|\alpha_1|,\dots, |\alpha_{n-1}|$ converge to $1$ as $\beta_1\nearrow n/k$. However, by \eqref{eq: product} we see
\begin{equation*}
    |\alpha_1|\cdots|\alpha_{n-1}||T|\to +\infty,\quad\textrm{as}\;\beta_1\nearrow n/k.
\end{equation*}
Thus we must have $T\to +\infty$ as $\beta_1\nearrow n/k$.
\end{proof}
\begin{prop}\label{prop: length to infty nk case}
As $\beta_1\nearrow n/k$, the length of the path on each fiber between the intersection point of the fiber with $Z_{n, k}$ and that of the fiber with $Z_{n, -k}$ tends to infinity. In other words, $Z_{n, -k}$ gets pushed--off to infinity as $\beta_1\nearrow n/k$ if we fix a base point on $Z_{n, k}$.
\end{prop}
\begin{proof}
Restricted to each fiber, by \eqref{eq: metric res}, $\eta=\varphi \sqrt{-1}\textrm{d}w\wedge\textrm{d}\bar{w}/|w|^2$ gives a metric
\begin{equation*}
    g=\frac{1}{2\varphi(\tau)}\textrm{d}\tau^2+2\varphi(\tau)\textrm{d}\theta^2.
\end{equation*}
Up to some constant, the distance between $\{\tau=1\}$ and $\{\tau=T\}$ is given by
\begin{equation}\label{eq: length int nk case}
\begin{aligned}
    &\int_1^T \frac{1}{\sqrt{\varphi(\tau)}}\;\textrm{d}\tau=\\
    &\int_1^T \frac{\tau^{\frac{n-1}{2}}}{\sqrt{\frac{n/k-\beta_1}{n+1}}\cdot\sqrt{(\tau-1)(\tau-\alpha_1)\cdots(\tau-\alpha_{n-1})(T-\tau)}}\;\textrm{d}\tau.
\end{aligned}    
\end{equation}
Recall in the proof of Proposition \ref{prop: T infty nk case}, we have shown 
\begin{equation}\label{eq: T asm behav}
    T\sim \frac{1+k\beta_1}{n-k\beta_1}\sim\frac{1+n}{k}\cdot\frac{1}{n/k-\beta_1},\quad \textrm{as}\;\beta_1\nearrow n/k.
\end{equation}
For any fixed $\epsilon>0$, consider 
\begin{equation*}
    \int_{T-\epsilon}^T \frac{\tau^{\frac{n-1}{2}}}{\sqrt{\frac{n/k-\beta_1}{n+1}}\cdot\sqrt{(\tau-1)(\tau-\alpha_1)\cdots(\tau-\alpha_{n-1})(T-\tau)}}\;\textrm{d}\tau.
\end{equation*}
We can find $\beta_1$ close to $n/k$ such that
\begin{equation*}
    T-\epsilon>\frac{1}{2}T.
\end{equation*}
Then in $[T-\epsilon, T]$, we have
\begin{align}
\begin{aligned}\label{eq: items in length int}
    &\tau^{\frac{n-1}{2}}>(T-\epsilon)^{\frac{n-1}{2}}>\left(\frac{1}{2}T\right)^{\frac{n-1}{2}},\\
    &\sqrt{(\tau-1)(\tau-\alpha_1)\cdots(
    \tau-\alpha_{n-1})}<\sqrt{2(\tau^n-1)}<\sqrt{2T^n}.
    \end{aligned}
\end{align}
By \eqref{eq: T asm behav} and \eqref{eq: items in length int}, there exists a constant $C>0$ that is independent of $\epsilon$ such that 
\begin{align*}
    &\int_{T-\epsilon}^T \frac{\tau^{\frac{n-1}{2}}}{\sqrt{\frac{n/k-\beta_1}{n+1}}\cdot\sqrt{(\tau-1)(\tau-\alpha_1)\cdots(\tau-\alpha_{n-1})(T-\tau)}}\;\textrm{d}\tau\\
    &>C\cdot \frac{T^{\frac{n-1}{2}}}{\sqrt{n/k-\beta_1}\cdot T^{\frac{n}{2}}} \int_{T-\epsilon}^T \frac{1}{\sqrt{T-\tau}}\;\textrm{d}\tau\\
    &>C\epsilon.
\end{align*}
Since $\epsilon$ is arbitrarily chosen, the integral in \eqref{eq: length int nk case} diverges as $\beta_1\nearrow n/k$. Thus, we have shown $Z_{n, -k}$ gets pushed-off to infinity as $\beta_1\nearrow n/k$ if we choose a base point on $Z_{n, k}$.
\end{proof}

\begin{prop}\label{prop: b2 limit nk case}
$\displaystyle\lim_{\beta_1\nearrow \frac{n}{k}} \beta_2(\beta_1)=\frac{1}{k}$.
\end{prop}
\begin{proof}
We calculate 
\begin{equation*}
    \frac{\textrm{d}\varphi}{\textrm{d}\tau}=\frac{1}{k}\frac{\tau^{n-2}}{(\tau^{n-1})^2}(n+\tau^n-1)+\frac{1}{n+1}(\beta_1-\frac{n}{k})\frac{\tau^{n-2}}{(\tau^{n-1})^2}(2\tau^{n+1}+n-1).
\end{equation*}
By $\varphi_\tau(T)=-\beta_2$ we have
\begin{equation}\label{eq: phi T general}
    \frac{1}{k T^n}(n+T^n-1)+\frac{1}{n+1}(\beta_1-\frac{n}{k})\frac{1}{T^n}(2T^{n+1}+n-1)=-\beta_2.
\end{equation}
Recall $\varphi(T)=0$, then we have
\begin{equation}\label{eq: phi T zero general}
    \frac{T^n-1}{k}+\frac{1}{n+1}(\beta_1-\frac{n}{k})(T^{n+1}-1)=0.
\end{equation}
Combining \eqref{eq: phi T general} and \eqref{eq: phi T zero general},
\begin{equation}
\label{Tbeta1beta2Eq}
    \Big(\frac{1}{k}-\beta_2\Big)T^n=\beta_1+\frac{1}{k}.
\end{equation}
Since $T\to \infty$ as $\beta_1\nearrow n/k$, there holds $\beta_2\to1/k$ as $\beta_1\nearrow n/k$.
\end{proof}

In Section \ref{sec: gh kee}, we use asymptotic behaviors discussed above to study the limit metric of $\eta$ as $\beta_1$ approaches $n/k$ or $0$. In the remainder of this section, we focus on the family of metrics $\xi$ that is parametrized by $\beta_2$. Inspired by Proposition \ref{prop: b2 limit nk case}, we study the asymptotic behaviors of $\xi$ as $\beta_2\nearrow 1/k$. 

\begin{prop}
\label{PropBeta2Beta1LimitAsymp}
When $\beta_2 \nearrow \frac{1}{k}$, as an analogue of \eqref{Tbeta1beta2Eq}, there holds
\begin{equation*}
    \left(\frac{1}{k}+\beta_1\right)t^n=\frac{1}{k}-\beta_2.
\end{equation*}
\end{prop}
\begin{proof}
Let $\varphi(\tau)$ and $t$ be as in \eqref{eq: bd cond nk b2 case} and \eqref{eq: beta2range}. We calculate
\begin{equation*}
    \frac{\textrm{d} \varphi}{\textrm{d}\tau} = \frac{1}{k}\frac{\tau^{n-2}}{(\tau^{n-1})^2}(\tau^n+n-1)-\frac{1}{n+1}\left(\frac{n}{k}+\beta_2\right)\frac{\tau^{n-2}}{(\tau^{n-1})^2}(2\tau^{n+1}+n-1).
\end{equation*}
By the fact that $\textrm{d}\varphi/\textrm{d}\tau(t)=\beta_1$, 
\begin{equation*}
    \beta_1 = \frac{1}{k t^n}(n+t^n-1)-\frac{1}{n+1}\left(\frac{n}{k}+\beta_2\right)\frac{1}{t^n}(2t^{n+1}+n-1).
\end{equation*}
Combining this with the fact that $\varphi(t)=0$, there holds
\begin{equation*}
\begin{aligned}
    \beta_1 t^n &= \frac{1}{k}(n+t^n-1)-\frac{2(t^n-1)}{k}-\frac{n}{k}-\beta_2\\
    & = -\frac{1}{k}t^n+\frac{1}{k}-\beta_2,
\end{aligned}
\end{equation*}
which implies
\begin{equation*}
    \left(\beta_1+\frac{1}{k}\right)t^n=\frac{1}{k}-\beta_2
\end{equation*}
as claimed.

\end{proof}

\begin{prop}\label{prop: length finite nkb2 case}
As $\displaystyle \beta_2\nearrow \frac{1}{k}$, $\displaystyle t = O\left(\left(\frac{k}{n+1}\right)^{\frac{1}{n}}\left(\frac{1}{k}-\beta_2\right)^{\frac{1}{n}}\right)$. In particular, $t$ tends to $0$ when $\beta_2\nearrow 1/k$. Moreover, the length of the path on each fiber between the intersection point of the fiber with $Z_{n, k}$ and that of the fiber with $Z_{n, -k}$ converges to a finite number as $\beta_2\nearrow 1/k$.
\end{prop}
\begin{proof}
Recall $t$ is a root of 
\begin{equation*}
    \varphi(\tau)=\frac{1}{k}\frac{\tau^n -1}{\tau^{n-1}}-\frac{1}{n+1}\left(\frac{n}{k}+\beta_2\right)\frac{\tau^{n+1}-1}{\tau^{n-1}}=0.
\end{equation*}
Direct calculations yield
\begin{equation*}
    \frac{\varphi(\tau)}{\tau-1}=\frac{1}{\tau^{n-1}}\left(-\frac{1}{n+1}\left(\frac{n}{k}+\beta_2\right)\tau^n+\frac{1/k-\beta_2}{n+1}(\tau^{n-1}+\cdots+\tau+1)\right).
\end{equation*}
Then it is easy to see that every root of $\varphi(\tau)$ except for $\tau=1$, denoted by $\alpha$, satisfies that $|\alpha| = O((k/(n+1))^{1/n}(1/k-\beta_2)^{1/n})$ when $\beta_2\nearrow 1/k$. In particular, as $\beta_2\nearrow 1/k$, all roots of $\varphi(\tau)=0$ except for $\tau=1$ converge to $0$. We conclude that $t\to 0$. 
To see that the length between $Z_{n, k}$ and $Z_{n, -k}$ converges to a finite number in the limit, one follows the arguments in the proof of Proposition \ref{prop: length to infty nk case} and Proposition \ref{prop: length finite k1 case}.
\end{proof}

\section{Large angle limits}\label{sec: gh kee}
In this section, we first study the Gromov--Hausdorff limit of the family of metrics $\eta$ when the parameter $\beta_1$ approaches $n/k$.

\begin{thm}\label{thm: metric limit nk case}
Fix an arbitrary base point $p$ on $Z_{n, k}$. As $\beta_1\nearrow n/k$, the K\"{a}hler--Einstein edge metric $(\mathbb{F}_{n, k}, \eta_{\beta_1}, p)$ converges in the pointed Gromov--Hausdorff sense to a Ricci-flat K\"{a}hler edge metric $(-kH_{\mathbb{P}^{n-1}}, \eta_\infty, p)$ described in \eqref{eq: limit eta nk case}. Moreover, away from $Z_{n, k}$, one has $\eta_{\beta_1}\xrightarrow{C^k_{\mathrm{loc}}}\eta_\infty$ for any $k\geq0$ on the total space of $-kH_{\mathbb{P}^{n-1}}$ (seen as an open subset in $\mathbb F_{n,k}$), and the limit metric $\eta_\infty$ coincides with the model metric $\omega_{\mathrm{eh}, n, k}$ defined in Section \ref{sec: model}.
\end{thm}
\begin{proof}
We use notation $\eta_{\beta_1}, \tau(\beta_1)$ and $\varphi(\beta_1)$ to emphasize the dependence of metrics and coordinates on $\beta_1$. Combining \eqref{eq: defi of tau and varphi nk case} and \eqref{eq: varphi general expression nk case}, we have
\begin{equation}\label{eq: relation s and tau}
    \frac{\textrm{d}s}{\textrm{d}\tau(\beta_1)}=\frac{1}{\varphi(\beta_1)}=\frac{\tau(\beta_1)^{n-1}}{\frac{1}{k}(\tau(\beta_1)^n-1)+\frac{\beta_1-n/k}{n+1}(\tau(\beta_1)^{n+1}-1)}.
\end{equation}

\begin{claim}\label{clm: tpinfty}
The pointwise limits of functions
\begin{align*}
    \tau_\infty(s):&=\lim_{\beta_1\nearrow n/k}\tau(\beta_1, s),\quad s\in(-\infty, +\infty),\\
    \varphi_\infty(s):&=\lim_{\beta_1\nearrow n/k}\varphi({\beta_1}, s),\quad s\in(-\infty, +\infty)
\end{align*}
exist. Moreover, $\tau_\infty(s)$ and $\varphi_\infty(s)$ are smooth in $s$.
\end{claim}

\begin{proof}[Proof of the Claim]
By the continuous dependence of ODEs in \eqref{eq: relation s and tau} on the parameter $\beta_1$, the pointwise limit function $\tau_\infty$ exists. Since $\tau(\beta_1, s)$ is smooth in $s$ for each $\beta_1$, $\tau_\infty(s)$ is smooth in $s$. By \eqref{eq: varphi general expression nk case} and the existence of $\tau_\infty(s)$, $\varphi_\infty(s)$ also exists. Moreover, $\varphi_\infty(s)$ is smooth in $s$ due to the smoothness of $\tau_\infty(s)$.
\end{proof}

Thanks to Claim \ref{clm: tpinfty}, $\tau_\infty$ and $\varphi_\infty$ satisfy
\begin{align}
\begin{aligned}\label{eq: tau infty and varphi infty}
    \frac{\textrm{d}s}{\textrm{d}\tau_{\infty
    }}&=\frac{\tau^{n-1}_\infty}{\frac{1}{k}(\tau_\infty^n-1)},\\
    \varphi_\infty(\tau_\infty)&=\frac{1}{k}\frac{\tau^n_\infty-1}{\tau^{n-1}_\infty}.
    \end{aligned}
\end{align}
Solving the first equation in \eqref{eq: tau infty and varphi infty}, we get
\begin{equation}\label{eq: taui ex formula}
    \tau_\infty(s) = (1+e^{(s-C)\frac{n}{k}})^{1/n}, \quad \text{for some constant}\; C.
\end{equation}
By considering a change of coordinate $w' = C' w$ for some appropriate $C'$ in \eqref{eq: def s}, we may choose $C$ in \eqref{eq: taui ex formula} to be $0$. Plugging \eqref{eq: taui ex formula} into the second equation in \eqref{eq: tau infty and varphi infty}, one finds 
\begin{equation}\label{eq: varphiinf ex for}
    \varphi_\infty(s)=\frac{1}{k}\cdot \frac{e^{s\cdot\frac{n}{k}}}{(e^{s\cdot\frac{n}{k}}+1)^{\frac{n-1}{n}}}.
\end{equation}
Plugging \eqref{eq: taui ex formula} and \eqref{eq: varphiinf ex for} into
\eqref{eq: eta in nk case}, we obtain the convergence of K\"{a}hler--Einstein edge metric $\eta_{\beta_1}$ on any compact subsets of $\mathbb{F}_{n, k}$ in every $C^k$-norm to the following metric:
\begin{align}
\begin{aligned}\label{eq: limit eta nk case}
    \eta_\infty:&=\lim_{\beta_1\nearrow n/k}\eta_{\beta_1}\\
    &=
    k (1+e^{s\cdot\frac{n}{k}})^{1/n} \pi_1^*\omega_{\operatorname{FS}} +
    \frac{1}{k}\cdot \frac{e^{s\cdot\frac{n}{k}}}{(e^{s\cdot\frac{n}{k}}+1)^{\frac{n-1}{n}}} \vphantom{\frac{\textrm{d} w}{w}}\left(\pi_2^*\omega_{\operatorname{Cyl}}+\sqrt{-1}\alpha\wedge\bar{\alpha}\right.\\
    &\left.+\sqrt{-1}\alpha\wedge\frac{\textrm{d}\bar{w}}{\bar{w}}+\sqrt{-1}\frac{\textrm{d}w}{w}\wedge\bar{\alpha}
    \right).
    \end{aligned}
\end{align}
Recall by \eqref{eq: ric curv gene}, the Ricci curvature of $\eta_{\beta_1}$ is given by $\lambda_{\beta_1}=n/k-\beta_1$, which converges to $0$ as $\beta_1 \nearrow n/k$. Thus, $\eta_\infty$ is a Ricci-flat K\"{a}hler edge metric on $-kH_{\mathbb{P}^{n-1}}$. $\eta_\infty$ has edge singularity of angle $2n\pi/k$ along $Z_{n, k}\subset -kH_{\mathbb{P}^{n-1}}$. Indeed, $\eta_\infty$ coincides with the model metric $\omega_{\mathrm{eh}, n, k}$ defined in Section \ref{sec: model}. To obtain the convergence in the pointed Gromov--Hausdorff sense, we first recall by Proposition \ref{prop: length to infty nk case} the distance between $Z_{n, -k}$ and $Z_{n, k}$ tends to infinity as $\beta_1 \nearrow n/k$. Once we choose a base point on $Z_{n, k}$. Since $\eta_{\beta_1}$ converges to $\eta_\infty$ on any compact geodesic balls centered at the base point, we conclude that $\eta_{\beta_1}$ converges in the pointed Gromov--Hausdorff sense to $\eta_\infty$ on $-kH_{\mathbb{P}^{n-1}}$.
\end{proof}

\begin{rmk}\label{rmk: eh as limit}
If we let $n=k=2$ in Theorem \ref{thm: metric limit nk case}, then by Remark \ref{rmk: eh spec case} we obtain in the limit the Eguchi--Hanson metric with parameter $\epsilon$ set as $1$ (see \eqref{eq: eh metric potential form}). In other words, the Eguchi--Hanson metric arises as the pointed Gromov--Hausdorff limit of \KE edge metrics $\eta_{\beta_1}$ when $\beta_1\nearrow 1$. This interesting observation has been conjectured in our previous work \cite[Remark 5.1]{RZ21} and provided some of the
motivation for the present article.
\end{rmk}


Next, we fix a base point on the infinity section $Z_{n, -k}$ to study the Gromov--Hausdorff limit of the family of \KE edge metrics $\xi$ on $\mathbb{F}_{n, k}$. In the limit, we obtain an orbifold \KE edge metric instead of the Ricci-flat edge metric obtained in Theorem \ref{thm: metric limit nk case}.

From now on, we use $\xi_{\beta_2}, \tau(\beta_2)$ and $\varphi(\beta_2)$ to emphasize the dependence of metrics and coordinates on $\beta_2$. We consider the case $\beta_2\nearrow 1/k$.




\begin{thm}\label{thm: conv mod2}
Fix an arbitrary base point $p$ on the infinity section $Z_{n, -k}$. As $\beta_2\nearrow 1/k$, the \KE edge metric $(\mathbb{F}_{n, k}, \xi_{\beta_2}, p)$ on $\mathbb{F}_{n, k}$ converges in the pointed Gromov--Hausdorff sense to an orbifold \KE edge metric $(\mathbb{P}^n(1,\dots,1,k), \xi_\infty, p)$ on the weighted projective space $\mathbb{P}^n(1,\dots,1,k)$ with an edge singularity of angle $2\pi/k$ along $Z_{n, -k}$. This limit metric coincides with the model metric $\omega_{\mathrm{orb}, n, k}$ defined in Definition \ref{def_ModOnWps}.
\end{thm}
\begin{proof}
By similar notation and calculations as in the proof of Theorem \ref{thm: metric limit nk case}, we have

\begin{equation*}
    \frac{\textrm{d}s}{\textrm{d}\tau(\beta_2)} = \frac{1}{\varphi(\beta_2)} = \frac{\tau(\beta_2)^{n-1}}{\frac{1}{k}(\tau(\beta_2)^n-1)-\frac{n/k+\beta_2}{n+1}(\tau(\beta_2)^{n+1}-1)}.
\end{equation*}
As $\beta_2\nearrow 1/k$, we have
\begin{equation}\label{eq_OdeTs}
\begin{aligned}
    \frac{\textrm{d}s}{\textrm{d}\tau_\infty} &= \frac{k}{\tau_\infty-\tau_\infty^2},\\
    \varphi_\infty(\tau_\infty) &= \frac{1}{k}(\tau_\infty-\tau_\infty^2). 
\end{aligned}    
\end{equation}
Solving \eqref{eq_OdeTs} and considering a change of coordinate $u'=C' u$ for some appropriate $C$, we have
\begin{align*}
    \tau_\infty(s)&=1-\frac{1}{e^{\frac{s}{k}}+1},\quad s\in(-\infty, +\infty),\\
    \varphi_\infty(s)&=\frac{1}{k}\frac{e^{s/k}}{(e^{s/k}+1)^2},\quad s\in(-\infty, +\infty).
\end{align*}
Thus, the limit metric on $\mathbb{F}_{n, k}$ is as follows:
\begin{align*}
    \tilde{\xi}_\infty&=k\frac{e^{s/k}}{e^{s/k}+1}\pi_1^*\omega_{\operatorname{FS}}+\frac{1}{k}\frac{e^{s/k}}{(e^{s/k}+1)^2}\left(\pi_2^* \omega_{\operatorname{Cyl}}+\sqrt{-1}\alpha\wedge\bar{\alpha}\right.\\
    &\left.+\sqrt{-1}\alpha\wedge\frac{\textrm{d}\bar{w}}{\bar{w}}+\sqrt{-1}\frac{\textrm{d}w}{w}\wedge\bar{\alpha}\right).
\end{align*}
Recall the Ricci curvature $\mu_{\beta_2}$ is given in \eqref{eq: beta2range} by $n/k+\beta_2$ and converges to $(n+1)/k$ in the limit. Thus $\tilde{\xi}_{\infty}$ has Ricci curvature $(n+1)/k$. Moreover, $\tilde{\xi}_{\infty}$ has an edge singularity of angle $2\pi/k$ along $Z_{n, -k}$. It degenerates on $Z_{n, k}$ since $\tau\equiv 0$ on $Z_{n, k}$. Indeed, $\tilde{\xi}_\infty$ coincides with the model metrics defined in Definition \ref{def_ModOnFnk}. Then by Definition \ref{def_ModOnWps}, we denote by $\xi_\infty$ the model metric on $\mathbb{P}^n(1,\dots, 1,k)$ that is the pull-back of $\tilde{\xi}_\infty$ under the blow up map.

We have shown that $\tilde{\xi}_\infty$ is the limit of $\xi_{\beta_2}$ as tensors in the pointwise smooth sense. Next, fix an arbitrary base point on $Z_{n, -k}$. By Proposition \ref{prop: length finite nkb2 case} and the local smooth convergence result, we conclude that $(\mathbb{P}^{n}(1,\dots,1,k), \xi_\infty)$ is the limit of $(\mathbb{F}_{n, k}, \xi_{\beta_2})$ in the pointed Gromov--Hausdorff sense when $\beta_2\nearrow 1/k$. Moreover, the limit metric coincides with the model metric $\omega_{\mathrm{orb}, n, k}$ defined in Section \ref{sec: model}. 
\end{proof}

As we pointed out in Section \ref{sec: asym of kee}, the family of metrics $\xi_{\beta_2}$ can be obtained by renormalizing the family of metrics $\eta_{\beta_1}$. Comparing Theorem \ref{thm: metric limit nk case} to Theorem \ref{thm: conv mod2}, we obtain different limit metrics for those two family of metrics. However, we show that after a proper normalization of $\xi_{\beta_2}$, we obtain the same limit metric for both $\xi_{\beta_2}$ and $\eta_{\beta_1}$. The normalization factor is actually given by Proposition \ref{prop: length finite nkb2 case}.

\begin{cor}\label{thm: ResB2}
Rescale the \KE edge metric $\xi_{\beta_2}$ by $((n+1)/k)^{1/n}/(1/k-\beta_2)^{1/n}$, then the normalized metric converges in the pointed Gromov--Hausdorff sense to a Ricci-flat metric on $-kH_{\mathbb{P}^{n-1}}$ when $\beta_2\nearrow 1/k$, where the base point is chosen from $Z_{n, k}$. See the proof for a more precise explanation. Moreover, this Ricci-flat metric coincides with the one obtained in Theorem \ref{thm: metric limit nk case}, i.e., the model metric $\omega_{\mathrm{eh}, n, k}$.
\end{cor}
\begin{proof}
Consider a change of coordinate
\begin{equation*}
    y({\beta_2}):=\left(\frac{n+1}{k}\right)^{\frac{1}{n}}\cdot \frac{\tau({\beta_2})}{(\frac{1}{k}-\beta_2)^{\frac{1}{n}}}.
\end{equation*}
By Proposition \ref{prop: length finite nkb2 case}, the interval of definition of $y({\beta_2})$ converges to $[1, +\infty]$ as $\beta_2\nearrow 1/k$. The rescaled metric reads 
\begin{align*}
    &\left(\frac{n+1}{k}\right)^{\frac{1}{n}}\cdot\frac{\xi_{\beta_2}}{(\frac{1}{k}-\beta_2)^{\frac{1}{n}}}\\
    &=ky\pi_1^*\omega_{\operatorname{FS}}+\left(\frac{n+1}{k}\right)^{\frac{1}{n}}\cdot\frac{\varphi({\beta_2})}{(\frac{1}{k}-\beta_2)^{\frac{1}{n}}}\left(\pi_2^*\omega_{\operatorname{Cyl}}+\sqrt{-1}\alpha\wedge\bar{\alpha}-\right.\\
    &\left.\sqrt{-1}\alpha\wedge\frac{\textrm{d}\bar{u}}{\bar{u}}-\sqrt{-1}\frac{\textrm{d}{u}}{{u}}\wedge\bar{\alpha} \right).
\end{align*}
Recall
\begin{equation}\label{eq: rescaled phi nk case}
\begin{aligned}
    &\left(\frac{n+1}{k}\right)^{\frac{1}{n}}\cdot\frac{\varphi({\beta_2})}{(\frac{1}{k}-\beta_2)^{\frac{1}{n}}}=\\
    &\left(\frac{n+1}{k}\right)^{\frac{1}{n}}\cdot\frac{\frac{1}{k}(\tau({\beta_2})^n-1)-\frac{1}{n+1}(\frac{n}{k}+\beta_2)(\tau({\beta_2})^{n+1}-1)}{(\tau({\beta_2})^{n-1})(\frac{1}{k}-\beta_2)^{\frac{1}{n}}}.
\end{aligned}
\end{equation}
Denote by $y$ the coordinate in the limit. Letting $\beta_2\nearrow 1/k$, the right hand side of \eqref{eq: rescaled phi nk case} converges to
\begin{equation*}
    \frac{y^n-1}{ky^{n-1}},\quad y\in[1, +\infty].
\end{equation*}
Solving 
\begin{equation*}
    \frac{\textrm{d}s}{\textrm{d}\tau_{\beta_2}}=\frac{1}{\varphi_{\beta_2}(\tau_{\beta_2})}\\
    \Rightarrow \frac{\textrm{d}s}{\textrm{d}y_{\beta_2}}\frac{1}{(\frac{1}{k}-\beta_2)^{\frac{1}{n}}}\cdot \left(\frac{n+1}{k}\right)^{\frac{1}{n}}=\frac{1}{\varphi_{\beta_2}(y_{\beta_2})},
\end{equation*}
we obtain in the limit
\begin{equation}\label{eq: y hacase}
    s=\frac{k}{n}\log\left(y^n-1\right),\quad y\in(1, +\infty).
\end{equation}
Thus the limit metric is given by
\begin{equation}\label{eq: lim metric nkb2 renor case}
\begin{aligned}
    \tilde{\xi}_\infty=ky\pi_1^*\omega_{\operatorname{FS}}+\frac{y^n-1}{ky^{n-1}}&\left(\pi_2^*\omega_{\operatorname{Cyl}}+\sqrt{-1}\alpha\wedge\bar{\alpha}-\right.\\
    &\left.\sqrt{-1}\alpha\wedge\frac{\textrm{d}\bar{u}}{\bar{u}}-\sqrt{-1}\frac{\textrm{d}{u}}{{u}}\wedge\bar{\alpha} \right),
\end{aligned}
\end{equation}
where $y$ and $s$ satisfy \eqref{eq: y hacase}. This limit metric is Ricci-flat. And it coincides with the limit metric in Theorem \ref{thm: metric limit nk case}, i.e., the model metric $\omega_{\mathrm{eh}, n, k}$ defined in Section \ref{sec: model}. Fix an arbitrary base point on $Z_{n, k}$. By Proposition \ref{prop: length finite nkb2 case}, the distance between $Z_{n, k}$ and $Z_{n, -k}$ tends to $+\infty$ in the limit under the renormalized metric. Thus, $Z_{n, -k}$ gets pushed-off to infinity in the limit. We obtain the pointed Gromov--Hausdorff convergence of $(\mathbb{F}_{n, k}, ((n+1)/k)^{1/n}\xi_{\beta_2}/(\frac{1}{k}-\beta_2)^{1/n})$
to $-kH_{\mathbb{P}^{n-1}}$ with the metric obtained in \eqref{eq: lim metric nkb2 renor case}.
\end{proof}

\section{Small angle limits
and fiberwise rescaling}
\label{Sec7}

In this section we first consider the $\beta_1\searrow 0$ case for $\eta_{\beta_1}$. 
\begin{thm}\label{thm: b1 0 case}
As $\beta_1$ tends to $0$, $(\mathbb{F}_{n, k}, \eta_{\beta_1})$ converges in the Gromov--Hausdorff sense to $(\mathbb{P}^{n-1}, k\omega_{\operatorname{FS}})$.
\end{thm}
\begin{proof}
As $\beta_1 \searrow 0$, by \eqref{eq: varphi general expression nk case} we have
\begin{align*}
    \varphi_0:=\lim_{\beta_1\searrow 0}\varphi_{\beta_1}&=\frac{1}{\tau^{n-1}}\left(\frac{1}{k}(\tau^n-1)-\frac{n}{k(n+1)}(\tau^{n+1}-1)\right)\\
    &=\frac{1}{k(n+1)}\cdot\frac{1}{\tau^{n-1}}(-n\tau^{n+1}+(n+1)\tau^n-1).
\end{align*}
Then we observe $\varphi_0$ does not have any root greater than $1$. Indeed, notice that
\begin{align*}
    -n\tau^{n+1}+(n+1)\tau^n-1&=(\tau-1)(1+\cdots+\tau^{n-1}-n\tau^n)\\
    &=(\tau-1)(1-\tau^n+\tau-\tau^n+\cdots+\tau^{n-1}-\tau^n).
\end{align*}
Thus $\varphi_0$ is always positive when $\tau>1$. However, combining this fact with \eqref{eq: varphi roots expre}, we conclude that 
\begin{equation}
\label{TbetatozeroEq}
\lim_{\beta_1\searrow 0}T=1.
\end{equation}
Since $\varphi(1)=\varphi(T)=0$, we have $\varphi({\beta_1})\to 0$ as $\beta_1 \searrow 0$. Since $\tau$ ranges from $1$ to $T$, by \eqref{eq: eta in nk case} we conclude that as $\beta_1\searrow 0$, $\eta_{\beta_1}$ converges to $k\pi_1^* \omega_{\operatorname{FS}}$. Thus we have shown $(\mathbb{F}_{n, k}, \eta_{\beta_1})$ converges in the Gromov--Hausdorff sense to $(\mathbb{P}^{n-1}, k\omega_{\operatorname{FS}})$ when $\beta_1\searrow 0$.
\end{proof}

Roughly speaking, Theorem \ref{thm: b1 0 case} says that as $\beta_1\searrow 0$, the fibers collapse to the zero section. This motivates us to rescale the metric $\eta_{\beta_1}$ along the fiber so that we can obtain a non-collapsed metric in the limit. We first need the following lemma.
\begin{lem}\label{lem: b1Tasy}
For $\beta_1>0$ and close to zero, $T=T(\beta_1)=1 + O(\beta_1).
$
\end{lem}
\begin{proof}
By Proposition \ref{prop: red to ode}, $T(\beta_1)$ is determined by $\beta_1$, and $T(\beta_1)$ is the first root of the polynomial in \eqref{eq: polyforT}. We rewrite the polynomial in \eqref{eq: polyforT} as \begin{equation*}
    P(\tau) = \frac{k\beta_1-n}{k(n+1)}(\tau-1)\left(\tau^n+\frac{k\beta_1+1}{k\beta_1-n}(\tau^{n-1}+\cdots+1)\right).
\end{equation*}
Letting $y=\tau-1$, 
\begin{equation}\label{eq: fapT}
    \begin{aligned}
    P(y) &= \frac{k\beta_1-n}{k(n+1)} y  \left((y+1)^n + \frac{k\beta_1+1}{k\beta_1-n}\big((y+1)^{n-1}+\cdots+(y+1)+1\big)\right)\\
    &= \frac{k\beta_1-n}{k(n+1)} y \bigg(y^n+\ldots+
    y\Big(n+\frac{k\beta_1+1}{k\beta_1-n}(1+\ldots+n-1)\Big)+ \frac{k(n+1)\beta_1}{k\beta_1-n}
    \bigg)\\
    &=
    \frac{k\beta_1-n}{k(n+1)} y  \left(y Q(y) + \frac{k(n+1)\beta_1}{k\beta_1-n}
    \right),
    \end{aligned}
\end{equation}
where $Q(y)$ is a polynomial of
degree $n-1$ whose coefficients
depend on $\beta_1$ and
whose constant
term is 
\begin{equation}
\label{QzeroEq}
Q(0)=n
+
\frac{k\beta_1+1}{k\beta_1-n}
\big(1+\ldots+n-1)
=
n\frac{(n-1)(k\beta_1+1)}{2(k\beta_1-n)}
=
\frac{n+1}2
\frac{nk\beta_1-n}{k\beta_1-n}
.
\end{equation}
By Proposition \ref{prop: red to ode}, $T-1$ is a root of the term in the parenthesis of the second equation in \eqref{eq: fapT}, i.e.,
$$
0=(T-1)Q(T-1) + \frac{k(n+1)\beta_1}{k\beta_1-n}.
$$
In particular, it 
follows that $Q(T-1)\neq0$
for small
enough $\beta_1$. Thus,
dividing
we obtain
$$
T-1=
\frac{k(n+1)\beta_1}{(-k\beta_1+n)Q(T-1)}.
$$
By  \eqref{TbetatozeroEq}
$\lim_{\beta_1\searrow0}
T=1$, and so  
$\lim_{\beta_1\searrow0}
Q(T-1)=(n+1)/2$
by
\eqref{QzeroEq}.
Altogether,
\begin{equation}
\label{Tminus1AsymEq}
T-1=\frac{2k}n\beta_1+o(\beta_1),
\end{equation}
as claimed.
\end{proof}

\begin{rmk}
The last display generalizes
\cite[(5.1)]{RZ21} from
the surface case $n=2$ to any dimension.
Note that in op. cit. $n$ corresponds
to our $k$
\end{rmk}

It follows that in the small
angle limit, both angles approach
zero at the same rate:
\begin{lem}
\label{lem: beta2beta1small}
For $\beta_1>0$ and close to zero, $\beta_2=\beta_2(\beta_1)=\beta_1 + O(\beta_1^2).
$
\end{lem}
\begin{proof}
Combining \eqref{TbetatozeroEq}
and \eqref{Tbeta1beta2Eq}
it follows that
$\lim_{\beta_1\searrow0}
\beta_2=0$.
Using this, and plugging \eqref{Tminus1AsymEq},
in \eqref{Tbeta1beta2Eq}
we find that
$$
\beta_1+\frac1k=\frac1k+2\beta_1-\beta_2+o(\beta_1),
$$
so $\beta_2=\beta_1+o(\beta_1)$,
and so bootstrapping we obtain
$\beta_2=\beta_1+O(\beta_1^2)$,
as claimed.\end{proof}

Lemma \ref{lem: beta2beta1small}
motivates treating the angles
$2\pi\beta_1$ and $2\pi\beta_2$ on 
the same footing in the small angle regime, so that it reasonable to 
hope that under some appropriate 
rescaling the fibers converge to 
cylinders, as in \cite{RZ20,RZ21}.
This is precisely what we prove next.

We change variable from $\tau\in(1, T)$ in \eqref{eq: eta in nk case} and \eqref{eq: varphi general expression nk case} to
\begin{equation}\label{eq: newcoord}
    x:= \frac{\tau-1-\frac{k\beta_1}{n}}{\frac{k\beta_1^2}{n}},
\end{equation}
with $x\in\left(-\frac{1}{\beta_1}, \frac{1}{\beta_1}+O(1)\right)$ by \eqref{Tminus1AsymEq}. Note that $x=0$ roughly corresponds to the middle section between $Z_{n, k}$ and $Z_{n, -k}$. By \eqref{eq: varphi general expression nk case} and \eqref{eq: newcoord},
\begin{equation}\label{eq: varphi in x}
    \varphi(x) = \frac{k}{2n}\beta_1^2 + \frac{k}{n}\beta_1^3 x + o(\beta_1^2), \quad x\in\left(-\frac{1}{\beta_1}, \frac{1}{\beta_1}+O(1)\right).
\end{equation}

Let $p\in \mathbb{F}_{n, k}$ be a fixed base point chosen from the section $\{x=0\}$, which will serve as the base point we use later for pointed Gromov--Hausdorff convergence.

To find a fiberwise-rescaled limit, we next rescale the metric $\eta_{\beta_1}$ in \eqref{eq: eta in nk case} along each fiber, i.e., we define
\begin{equation}\label{eq: fibrescal}
    \widetilde{\eta_{\beta_1}}:=k\tau \pi_1^*\omega_{\operatorname{FS}}+\frac{1}{\beta_1^2}\varphi\pi_2^*\omega_{\operatorname{Cyl}} + \varphi(\sqrt{-1}\alpha\wedge\bar{\alpha}+\sqrt{-1}\alpha\wedge \overline{\textrm{d}w/w}+\sqrt{-1}{\textrm{d}w/w}\wedge\bar{\alpha}).
\end{equation}

\begin{rmk}
This fiberwise rescaled metric is no longer K\"ahler. Indeed, since the metric $\eta_{\beta_1}$ in \eqref{eq: eta in nk case} is K\"{a}hler, i.e., $\textrm{d} \eta_{\beta_1}=0$, there holds
\begin{equation*}
\begin{aligned}
\textrm{d}\widetilde{\eta_{\beta_1}}&=\textrm{d}\left( \eta_{\beta_1} - \left(1-\frac{1}{\beta_1^2}\right) \varphi \pi_2^*\omega_{\operatorname{Cyl}}\right)\\
&= - \left(1-\frac{1}{\beta_1^2}\right) \textrm{d}\varphi \wedge \pi^*_2 \omega_{\operatorname{Cyl}}\\
&\neq 0.
\end{aligned}
\end{equation*}
\end{rmk}

\begin{thm}\label{thm: ResB1Case}
As $\beta_1\searrow 0$, $(\mathbb{F}_{n, k}, \widetilde{\eta_{\beta_1}}, p)$ converges in the pointed Gromov--Hausdorff sense to $(\mathbb{P}^{n-1}\times \mathbb{C}^*, \frac{k}{n}(n\pi_1^*\omega_{\operatorname{FS}}+\pi_2^*\omega_{\operatorname{Cyl}}), p)$.
\end{thm}

\begin{proof}
We first show on compact subsets, the following pointwise convergence holds:
\begin{claim}
The restriction of $\widetilde{\eta_{\beta_1}}$ to a fiber converges to a cylindrical metric pointwise on compact subsets. More precisely,
\begin{equation*}
    \lim_{\beta_1\searrow 0} \frac{1}{\beta_1^2}\varphi \pi_2^* \omega_{\operatorname{Cyl}} = \frac{k}{n} \pi_2^* \omega_{\operatorname{Cyl}}.
\end{equation*}
\end{claim}

\begin{proof}[Proof of the Claim]
As shown in \eqref{eq: metric res}, the restriction of ${\eta_{\beta_1}}$ is given by 
\begin{equation*}
    \frac{1}{2\varphi(\tau)}\textrm{d}\tau^2 + 2\varphi(\tau) \textrm{d}\theta^2,
\end{equation*}
thus the restriction of $\widetilde{\eta_{\beta_1}}$ to a fiber, using the new coordinates \eqref{eq: newcoord} is given by
\begin{equation}\label{eq: ResVarphiFib}
    \frac{k^2\beta_1^2}{2n^2\varphi(x)}\textrm{d}x^2 + \frac{2\varphi(x)}{\beta_1^2}\textrm{d}\theta^2.
\end{equation}
As $\beta_1 \searrow 0$, by \eqref{eq: varphi in x}, \eqref{eq: ResVarphiFib} converges pointwise on compact subsets to 
\begin{equation*}
    \frac{k}{n}\textrm{d}x^2 + \frac{k}{n}\textrm{d}\theta^2 = \frac{k}{n}\omega_{\operatorname{Cyl}},
\end{equation*}
as claimed.
\end{proof}
By the collapsing arguments in the proof of Theorem \ref{thm: b1 0 case} and the claim above, we have shown $\widetilde{\eta_{\beta_1}}$ converges pointwise to $\frac{k}{n}(n\pi_1^* \omega_{\operatorname{FS}}+\pi_2^* \omega_{\operatorname{Cyl}})$ on compact subsets as $\beta_1\searrow 0$. It remains to prove the pointed Gromov--Hausdorff convergence. Indeed, by arguments in the proof of Proposition \ref{prop: length to infty nk case}, the distance between $Z_{n, k}$ and $\{x=0\}$ and the distance between $Z_{n, -k}$ and $\{x=0\}$ tend to infinity under the metric $\widetilde{\eta_{\beta_1}}$. Thus in the limit $\beta_1\searrow 0$, we get the product differential structure on $\mathbb{P}^{n-1}\times \mathbb{C}^*$ as claimed. Choosing the point $p$ as the base point, the pointwise convergence result implies that
\begin{equation*}
    \lim_{\beta_1\searrow 0} \widetilde{\eta_{\beta_1}} = \frac{k}{n}(n\pi_1^* \omega_{\operatorname{FS}}+\pi_2^* \omega_{\operatorname{Cyl}})
\end{equation*}
in the pointed Gromov--Hausdorff sense.
\end{proof}

The $\beta_2\searrow 0$ case for $\xi_{\beta_2}$ is similar to Theorem \ref{thm: b1 0 case} and Theorem \ref{thm: ResB1Case}.

The asymptotic behaviors of $t(\beta_2)$ and $\beta_1(\beta_2)$ when $\beta_2 \searrow 0$ are similar to those described in Lemma \ref{lem: b1Tasy} and Lemma \ref{lem: beta2beta1small}. We collected them as follows.
\begin{lem}\label{lem: b2tasy}
For $\beta_2>0$ and close to zero, 
\begin{equation*}
    t = t(\beta_1) = 1-\frac{2k}{n}\beta_2+o(\beta_2).
\end{equation*}
\end{lem}

\begin{lem}\label{lem: b1b2s}
For $\beta_2>0$ and close to zero, 
\begin{equation*}
    \beta_1 = \beta_1(\beta_2)=\beta_2+O(\beta_2^2).
\end{equation*}
\end{lem}
Now we state the non-rescaling limit of $\xi_{\beta_2}$ as $\beta_2\searrow 0$.

\begin{thm}\label{thm: b2 0 case}
As $\beta_2$ tends to $0$, $(\mathbb{F}_{n, k}, \xi_{\beta_2})$ converges in the Gromov--Hausdorff sense to $(\mathbb{P}^{n-1}, k\omega_{\operatorname{FS}})$.
\end{thm}

\begin{proof}
By Proposition \ref{prop: bd cond b2 case} and Proposition \ref{prop: red to ode b2 case}, $t$ tends to $1$ as $\beta_2\searrow 0$. The remaining proof is similar to that of Theorem \ref{thm: b1 0 case}.
\end{proof}

To obtain a non-collapsed metric in the limit, we consider rescaling $\xi_{\beta_2}$ in the way of \eqref{eq: fibrescal} and denote the rescaled metric by $\widetilde{\xi_{\beta_2}}$:

\begin{equation*}
    \widetilde{\xi_{\beta_2}}:=k\tau \pi_1^*\omega_{\operatorname{FS}}+\frac{1}{\beta_2^2}\varphi\pi_2^*\omega_{\operatorname{Cyl}} + \varphi(\sqrt{-1}\alpha\wedge\bar{\alpha}-\sqrt{-1}\alpha\wedge \overline{\textrm{d}u/u}-\sqrt{-1}{\textrm{d}u/u}\wedge\bar{\alpha}).
\end{equation*}

Moreover, we consider a change of variable as \eqref{eq: newcoord}:
\begin{equation*}
    u := \frac{\tau-1+\frac{k}{n}\beta_2}{\frac{k}{n}\beta_2^2}.
\end{equation*}
As before, we choose a fixed base point from the section $\{u=0\}$.

\begin{thm}\label{thm: ResB2Zero}
As $\beta_2\searrow 0$, $(\mathbb{F}_{n, k}, \widetilde{\xi_{\beta_2}}, q)$ converges in the pointed Gromov--Hausdorff sense to $(\mathbb{P}^{n-1}\times \mathbb{C}^*, \frac{k}{n}(n\pi_1^*\omega_{\operatorname{FS}}-\pi_2^*\omega_{\operatorname{Cyl}}), q)$. 
\end{thm}

\begin{proof}
By Lemma \ref{lem: b2tasy} and Lemma \ref{lem: b1b2s}, we have similar asymptotic behaviors as Lemma \ref{lem: b1Tasy} and Lemma \ref{lem: beta2beta1small} in the $\beta_2$ case. Then we can apply similar arguments as in the proof of Theorem \ref{thm: ResB1Case}.
\end{proof}

\appendix

\section{A brief review on Eguchi--Hanson metrics}\label{app: review EH}
In this section, we give a brief review of the construction of Eguchi--Hanson metrics \cite{EH79}. They are Ricci-flat K\"{a}hler metrics defined on the total space of the line bundle $-2H_{\mathbb{P}^1}$. 

For $(x_1+\sqrt{-1}y_1, x_2+\sqrt{-1}y_2)\in\mathbb{C}^2$, the Hopf coordinates are defined as:
\begin{align*}
    x_1+\sqrt{-1}y_1&=\xi\cos\frac{\theta}{2} e^{\frac{\sqrt{-1}}{2}(\psi+\phi)},\\
    x_2+\sqrt{-1}y_2&=\xi\sin\frac{\theta}{2} e^{\frac{\sqrt{-1}}{2}(\psi-\phi)},
\end{align*}
where $\xi\geq 0$, $\theta\in[0, \pi]$, $\psi\in[0, 4\pi]$ and $\phi\in[0, 2\pi]$. Define one-forms on $\mathbb{C}^2$ by
\begin{align}
\begin{aligned}\label{eq: defi of sigma}
    \sigma_1:&=\frac{1}{\xi^2}(x_1\textrm{d}y_2-y_2\textrm{d}x_1+y_1\textrm{d}x_2-x_2\textrm{d}y_1)=\frac{1}{2}(\sin\psi\textrm{d}\theta-\sin\theta\cos\psi\textrm{d}\phi),\\
    \sigma_2:&=\frac{1}{\xi^2}(y_1\textrm{d}y_2-y_2\textrm{d}y_1+x_2\textrm{d}x_1-x_1\textrm{d}x_2)=\frac{1}{2}(-\cos\psi\textrm{d}\theta-\sin\theta\sin\psi\textrm{d}\phi),\\
    \sigma_3:&=\frac{1}{\xi^2}(x_2\textrm{d}y_2-y_1\textrm{d}x_2+x_1\textrm{d}y_1-y_1\textrm{d}x_1)=\frac{1}{2}(\textrm{d}\psi+\cos\theta\textrm{d}\phi).
\end{aligned}
\end{align}
Direct calculations yield:
\begin{align}
\begin{aligned}\label{eq: sigma cyclic}
    \textrm{d}\sigma_1=2\sigma_2\wedge\sigma_3,\\
    \textrm{d}\sigma_2=2\sigma_3\wedge\sigma_1,\\
    \textrm{d}\sigma_3=2\sigma_1\wedge\sigma_2.
\end{aligned}
\end{align}
The standard Euclidean metric on $\mathbb{C}^2$ can be written as
\begin{equation*}
    \textrm{d}x_1^2+\textrm{d}y_1^2+\textrm{d}x_2^2+\textrm{d}y_2^2=\textrm{d}\xi^2+\xi^2(\sigma_1^2+\sigma_2^2+\sigma_3^2).
\end{equation*}
The Eguchi--Hanson metric with parameter $\epsilon>0$ is defined by
\begin{equation}\label{eq: EH epsilon}
    g_{\operatorname{EH}, \epsilon}:=\left(1-\frac{\epsilon^4}{\xi^4}\right)^{-1}\textrm{d}\xi^2+\xi^2\left(\sigma_1^2+\sigma_2^2+\left(1-\frac{\epsilon^4}{\xi^4}\right)\sigma_3^2\right),\quad \xi\geq \epsilon.
\end{equation}
\begin{prop}\label{prop: eh ricflat}
Eguchi--Hanson metrics defined in \eqref{eq: EH epsilon} are Ricci-flat.
\end{prop}
\begin{proof}
We provide a proof by directly calculating connection forms and curvature forms of the metric. See Remark \ref{rmk: ricci flat using kahler form} for another proof using K\"{a}hler forms of Eguchi--Hanson metrics. Consider a change of variable $\zeta=\zeta(\xi)$ such that $\textrm{d}\zeta=(1-(\epsilon/\xi)^4)^{-1/2}d\xi$. Then we can write \eqref{eq: EH epsilon} in the form
\begin{equation*}
    g_{\operatorname{EH}, \epsilon}=\textrm{d}\zeta^2+f^2(\zeta)(\sigma_1^2+\sigma_2^2+g^2(\zeta)\sigma_3^2),
\end{equation*}
where $f=\xi$ and $g=(1-(\epsilon/\xi)^4)^{1/2}$. Consider an orthonormal basis
\begin{equation*}
    (\omega^0, \omega^1, \omega^2, \omega^3)=(\textrm{d}\zeta, fg\sigma_3, f\sigma_1, f\sigma_2).
\end{equation*}
Then $g_{\operatorname{EH}, \epsilon}=\sum_{i=0}^3(\omega^i)^2$ and $\{\omega^i\}_{i=0}^3$ satisfy the following equations:
\begin{align*}
    &\textrm{d}\omega^i=\omega^j\wedge\omega_{j}^i,\quad \text{for}\;i=0, 1,2,3,\\
    &\omega_i^j+\omega_j^i=0,\quad \text{for}\;i, j=0, 1,2,3,
\end{align*}
where $\{\omega_i^j\}_{i, j=0}^3$ are connection forms with respect to $\{\omega^i\}_{i=0}^3$. Next let us determine connection forms. For $i=1$, we have
\begin{equation*}
    \textrm{d}\omega^1=\frac{f'g+fg'}{fg}\omega^0\wedge\omega^1+\frac{2g}{f}\omega^2\wedge\omega^3,
\end{equation*}
where $f'$ and $g'$ denote the derivative with respect to $\zeta$. Without loss of generality, we let
\begin{align}
\begin{aligned}\label{eq:omega sub 1}
    \omega_0^1&=\frac{f'g+fg'}{fg}\omega^1,\\
    \omega_2^1&=\frac{g}{f}\omega^3,\\
    \omega_3^1&=-\frac{g}{f}\omega^2.
    \end{aligned}
\end{align}
It remains to find $\omega_0^2$, $\omega_0^3$ and $\omega_2^3$ due to the skew-symmetry of connection forms. By similar calculations we have
\begin{align}
\begin{aligned}\label{eq:d of omega others}
    \textrm{d}\omega^2=\frac{f'}{f}\omega^0\wedge\omega^2+\frac{2}{fg}\omega^3\wedge\omega^1,\\
    \textrm{d}\omega^3=\frac{f'}{f}\omega^0\wedge\omega^3+\frac{2}{fg}\omega^1\wedge\omega^2.
\end{aligned}
\end{align}
Thus, combining $\eqref{eq:omega sub 1}$ and \eqref{eq:d of omega others} we obtain
\begin{align}
\begin{aligned}\label{eq: omega others}
\omega_0^2&=\frac{f'}{f}\omega^2,\\
\omega_0^3&=\frac{f'}{f}\omega^3,\\
\omega_2^3&=\frac{g^2-2}{fg}.
\end{aligned}
\end{align}
Notice that 
\begin{align}
\begin{aligned}\label{eq: f and g}
    f'&=\frac{\textrm{d}\xi}{\textrm{d}\zeta}=(1-(\epsilon/\xi)^4)^{1/2}=g,\\
    g'&=2\epsilon^4f^{-5}.
    \end{aligned}
\end{align}
Combining \eqref{eq:omega sub 1}, \eqref{eq: omega others} and \eqref{eq: f and g} we see
\begin{align}
\begin{aligned}\label{eq: connection form selfdual}
    \omega_0^1&=-\omega_2^3,\\
    \omega_2^1&=\omega_0^3,\\
    \omega_3^1&=-\omega_0^2.
\end{aligned}
\end{align}
By \eqref{eq: connection form selfdual} and the fact that $R_i^j=\textrm{d}\omega_i^j-\omega_i^k\wedge\omega_k^j$, for $i, j=0, 1,2,3$, we obtain that the curvature forms also satisfy 
\begin{align*}
    R_0^1&=-R_2^3,\\
    R_2^1&=R_0^3,\\
    R_3^1&=-R_0^2.
\end{align*}
Finally, the Ricci-flatness comes from the formula $\operatorname{Ric}_{ij}=\sum_{k=0}^3R_{ikj}^{k}$ and the first Bianchi identity.
\end{proof}
Introduce 
\begin{equation*}
    r^4=\xi^4-\epsilon^4, \quad \xi\geq\epsilon,
\end{equation*}
then \eqref{eq: EH epsilon} can be written as
\begin{equation}\label{eq: eh metric in r}
    g_{\operatorname{EH}, \epsilon}=\frac{r^2}{(\epsilon^4+r^4)^{\frac{1}{2}}}(\textrm{d}r^2+r^2\sigma_3^2)+(\epsilon^4+r^4)^{\frac{1}{2}}(\sigma_1^2+\sigma_2^2),\quad r\geq 0.
\end{equation}
From \eqref{eq: eh metric in r} we are able to convert Eguchi--Hanson metrics into complex form by letting $(z_1, z_2)\in\mathbb{C}^2$ satisfy
\begin{align}
\begin{aligned}\label{complex coordinates for EH}
    z_1&=r\cos\frac{\theta}{2}e^{\frac{\sqrt{-1}}{2}(\psi+\phi)},\\
    z_2&=r\sin\frac{\theta}{2}e^{\frac{\sqrt{-1}}{2}(\psi-\phi)}.
    \end{aligned}
\end{align}
Denote by $\omega_{\operatorname{EH}, \epsilon}$ the K\"{a}hler form corresponding to $g_{\operatorname{EH}, \epsilon}$. Then by \eqref{eq: eh metric in r} we have in complex coordinates,
\begin{align}
    \omega_{\operatorname{EH}, \epsilon}
    &=\sqrt{-1}\partial\bar{\partial}[\sqrt{r^4+\epsilon^4}+\log r^2-\log(\epsilon^2+\sqrt{r^4+\epsilon^4})]\label{eq: eh metric potential form}\\
   &=\frac{\sqrt{-1}r^2}{\sqrt{r^4+\epsilon^4}}(\textrm{d}z_1\wedge\textrm{d}\bar{z}_1+\textrm{d}z_2\wedge\textrm{d}\bar{z}_2)+\frac{\epsilon^4}{\sqrt{r^4+\epsilon^4}}\sqrt{-1}\partial\bar{\partial}\log(r^2).\label{eq: eh metric kahler form}
\end{align}

\begin{rmk}\label{rmk: ricci flat using kahler form}
By calculating the Ricci form of $\omega_{\operatorname{EH}, \epsilon}$ as in \eqref{eq: eh metric kahler form}, we can also derive the Ricci-flatness of Eguchi--Hanson metrics. Indeed, from \eqref{eq: eh metric kahler form} we calculate
\begin{align*}
    \operatorname{Ric}\omega_{\operatorname{EH}, \epsilon}&=\operatorname{Ric}\sqrt{-1}\Bigg(
    \left[\frac{r^2}{\sqrt{r^4+\epsilon^4}}+\frac{\epsilon^4|z_2|^2}{\sqrt{r^4+\epsilon^4}r^4}\right]\textrm{d}z_1\wedge\textrm{d}\bar{z}_1\\
    &-\frac{\epsilon^4z_2\bar{z}_1}{\sqrt{r^4+\epsilon^4}r^4}\textrm{d}z_1\wedge\textrm{d}\bar{z}_2-\frac{\epsilon^4z_1\bar{z}_2}{\sqrt{r^4+\epsilon^4}r^4}\textrm{d}z_2\wedge\textrm{d}\bar{z}_1\\
    &+\left[\frac{r^2}{\sqrt{r^4+\epsilon^4}}+\frac{\epsilon^4|z_1|^2}{\sqrt{r^4+\epsilon^4}r^4}\right]\textrm{d}z_2\wedge\textrm{d}\bar{z}_2 \Bigg)\\
    &=-\sqrt{-1}\partial\bar{\partial}\log\left(
    \left[\frac{r^2}{\sqrt{r^4+\epsilon^4}}+\frac{\epsilon^4|z_2|^2}{\sqrt{r^4+\epsilon^4}r^4}\right]^2\right.\\
    &-\left.\frac{\epsilon^4|z_1|^2|z_2|^2}{(r^4+\epsilon^4)r^8}
    \right)\\
    &=-\sqrt{-1}\partial\bar{\partial}\log1\\
    &=0.
\end{align*}
\end{rmk}
From \eqref{eq: eh metric in r} one finds that $g_{\operatorname{EH}, \epsilon}$ is defined on $\mathbb{C}^2$ with possible singularity at $r=0$. Since $g_{\operatorname{EH},\epsilon}$ is invariant under the antipodal reflection, we have an induced metric on $(\mathbb{C}^2\setminus\{0\})/\mathbb{Z}_2$ that admits no singularity. Consider the blow-up of $(\mathbb{C}^2\setminus\{0\})/\mathbb{Z}_2$ at the origin, which is biholomorphic to the total space of the line bundle $-2H_{\mathbb{P}^1}$, then a calculation in \eqref{eq: sigma1 plus 2} below shows that $\sigma_1^2+\sigma_2^2$ is the pull-back of the Fubini--Study metric from the exceptional divisor. Hence $g_{\operatorname{EH}, \epsilon}$ extends to a metric on the total space $-2H_{\mathbb{P}^1}$ by letting $r=0$ when restricting $g_{\operatorname{EH}, \epsilon}$ to the exceptional divisor.

\section{Eguchi--Hanson metrics as Gromov--Hausdorff limits of K\"{a}hler--Einstein edge metrics}\label{app: eh another discussion}
In this section we give a direct proof of a special case of Theorem \ref{thm: metric limit nk case} when $n=k=2$. We already know we will obtain Eguchi--Hanson metrics in the limit. 

Recall in \eqref{eq: eta in nk case} we denote by $\eta$ a K\"{a}hler--Einstein edge metric on Calabi--Hirzebruch manifolds $\mathbb{F}_{n, k}$ that has the following form:
\begin{equation}\label{eq: eta origin expression}
    \eta=k\tau \pi_1^*\omega_{\operatorname{FS}}+\varphi\left(\pi_2^*\omega_{\operatorname{Cyl}}+\sqrt{-1}\alpha\wedge\bar{\alpha}+\sqrt{-1}\alpha\wedge\frac{\textrm{d}\bar{w}}{\bar{w}}+\sqrt{-1}\frac{\textrm{d}w}{w}\wedge\bar{\alpha}\right),
\end{equation}
where $\pi_1$, $\pi_2$ and $\alpha$ are defined below \eqref{eq: eta in nk case}.

From now on, we assume $k=2$ and $n=2$, i.e., consider the second Hirzebruch surface $\mathbb{F}_2$. To build a connection between $g_{\operatorname{EH}, \epsilon}$ and the K\"{a}hler edge metric $\eta$ on $\mathbb{F}_2$, we first write $\eta$ in terms of one forms introduced in \eqref{eq: defi of sigma}.

Consider a change of coordinate $w=v^2$. The reason to do this is that $w$ is the coordinate along each fiber of the line bundle $-2H_{\mathbb{P}^1}$. Recall \eqref{complex coordinates for EH}, then we have the following correspondence:
\begin{align}
\begin{aligned}\label{eq: coordinates uz}
    z_1=vz&=r\cos\frac{\theta}{2}e^{\frac{\sqrt{-1}}{2}(\psi+\phi)},\\
    z_2=v&=r\sin\frac{\theta}{2}e^{\frac{\sqrt{-1}}{2}(\psi-\phi)}.
    \end{aligned}
\end{align}
In particular, 
\begin{equation}
\begin{aligned}\label{eq: relation r u z}
    r^2&=|v|^2(1+|z|^2)\\
    &=|w|(1+|z|^2).
\end{aligned}
\end{equation}
By the definition of $\alpha$ in \eqref{eq: eta in nk case}, $(1, 1)$-forms that appear in $\eta$ are
\begin{equation*}
    \frac{\textrm{d}z\wedge \textrm{d}\bar{z}}{(1+|z|^2)^2},\; \alpha\wedge\bar{\alpha},\; \frac{4\textrm{d}v\wedge\textrm{d}\bar{v}}{|v|^2},\;\alpha\wedge\frac{2\textrm{d}\bar{v}}{\bar{v}},\;\frac{2\textrm{d}v}{v}\wedge\bar{\alpha}.
\end{equation*}

Let us first calculate $\sigma_1^2+\sigma_2^2$ in terms of the coordinate $z$ and $v$. By \eqref{eq: defi of sigma} and \eqref{eq: coordinates uz} we have
\begin{align}
\begin{aligned}\label{eq: sigma1 plus 2}
    \sigma_1^2+\sigma_2^2&=\frac{1}{4}\textrm{d}\theta^2+\frac{1}{4}\sin^2\theta\textrm{d}\phi^2\\
    &=\frac{1}{4}\left(-\frac{|z|\textrm{d}z}{(1+|z|^2)z}-\frac{|z|\textrm{d}\bar{z}}{(1+|z|^2)\bar{z}}\right)^2+\\
    &\frac{1}{4}\left(\frac{2|z|}{1+|z|^2}\right)^2\cdot\left(\frac{1}{2\sqrt{-1}}\left(\frac{\textrm{d}z}{z}-\frac{\textrm{d}\bar{z}}{\bar{z}}\right)\right)^2\\
    &=\operatorname{Re}\frac{\textrm{d}z\otimes\textrm{d}\bar{z}}{(1+|z|^2)^2}.
\end{aligned}
\end{align}

For $\textrm{d}r^2$ and $\sigma_3$, By \eqref{eq: defi of sigma} and \eqref{eq: coordinates uz} we have
\begin{align}
\begin{aligned}\label{eq:r and sigma3}
    4r^2\textrm{d}r^2&=\frac{r^4}{4}\left(\alpha+\bar{\alpha}+\frac{2\textrm{d}v}{v}+\frac{2\textrm{d}\bar{v}}{\bar{v}}\right)^2,\\
    \sigma_3^2&=-\frac{1}{16}\left(\alpha-\bar{\alpha}+\frac{2\textrm{d}v}{v}-\frac{2d\bar{v}}{\bar{v}}\right)^2.\\
    \textrm{d}r^2+r^2\sigma_3^2&=\frac{r^2}{4}\operatorname{Re}\left(\frac{4\textrm{d}v\otimes\textrm{d}\bar{v}}{|v|^2}+\alpha\otimes\bar{\alpha}+\alpha\otimes\frac{2\textrm{d}\bar{v}}{\bar{v}}+\frac{2\textrm{d}v}{v}\otimes\bar{\alpha}\right).
    \end{aligned}
\end{align}
Denote by $g_\eta$ the corresponding Riemannian metric on $-2H_{\mathbb{P}^1}$ with respect to $\eta$. Then combining \eqref{eq: eta origin expression}, \eqref{eq: sigma1 plus 2} and \eqref{eq:r and sigma3} we obtain
\begin{equation}\label{eq: eta in limit}
    g_\eta=2\tau(\sigma_1^2+\sigma_2^2)+\varphi\cdot \frac{4}{r^2}(\textrm{d}r^2+r^2\sigma_3^2).
\end{equation}

For $\tau$ and $\varphi$ in \eqref{eq: eta in limit}, results in Section \ref{sec: general} apply after we fix $n=k=2$ there. A key feature when $n=2$ is the right hand side in \eqref{eq: varphi general expression nk case} is a cubic polynomial, which is easy to handle. In other words, we will be able to derive more precise dependence of $T$ and $\beta_2$ on $\beta_1$ comparing to results in Proposition \ref{prop: T infty nk case} and Proposition \ref{prop: b2 limit nk case}. Indeed, this was done in the work of \cite{RZ21}. In this and the next section, we fix $n=2$ and make use of several results obtained in \cite{RZ21}. 

Since $n=k=2$, we find $\beta_1\in(0, 1)$. We will study the asymptotic behaviors of \KE edge metrics $\eta$ when $\beta_1\to 1$. Recall \cite[(4.20)]{RZ21}
\begin{equation*}
    \beta_2=\frac{1}{4}(2\beta_1+\sqrt{3(3-2\beta_1)(1+2\beta_1)}-3).
\end{equation*}
So we have $\beta_2\to \frac{1}{2}$ as $\beta_1\to 1$. Recall, $\tau$ ranges from $[1, T]$ and $T$ is given by \cite[(5.1)]{RZ21}
\begin{equation*}
    T=1+3\frac{\sqrt{1+ \frac{4}{3}\beta_1 -\frac{4}{3}\beta_1^2}+2\beta_1-1}{4-4\beta_1}.
\end{equation*}
Thus, $T\to +\infty$ as $\beta_1\to 1$. Moreover, recall in \eqref{eq: varphi general expression nk case} $\varphi(\tau)$ is given by
\begin{align}
    \varphi(\tau)&=\frac{1}{2}\frac{\tau^2-1}{\tau}+\frac{1}{3}(\beta_1-1)\frac{\tau^3-1}{\tau} \label{eq:varphi and tau}\\
    &=\frac{1}{3}(\beta_1-1)(\tau-1)(\tau-\alpha_1)(\tau-T)/\tau, \quad\textrm{for}\;\tau\in[1, T],\notag
\end{align}
where $\alpha_1$ is given in \cite[(5.2)]{RZ21} by $\displaystyle \alpha_1=1+3\frac{-\sqrt{1+\frac{4}{3}\beta_1-\frac{4}{3}\beta_1^2}+2\beta_1-1}{4-4\beta_1}$ and $\alpha_1$ tends to $-1$ as $\beta_1\to 1$.
Below we first show as $\beta_1\to 1$, the divisor $Z_{-2}$ gets pushed-off to infinity. This result is a special case of Proposition \ref{prop: length to infty nk case}, and for the reader's convenience we include a proof here.
\begin{prop}\label{prop: length to infty n2 case}
The length of the path on each fiber between the intersection point of the fiber with $Z_2$ and that of the fiber with $Z_{-2}$ tends to infinity as $\beta_1\to 1$.
\end{prop}
\begin{proof}
Restricted to the fiber $\{z=0\}$, by \eqref{eq: eta origin expression} $\eta=\varphi\sqrt{-1}\textrm{d}w\wedge\textrm{d}\bar{w}/{|w|^2}$ gives a metric
\begin{equation*}
    g=\frac{1}{2\varphi(\tau)}\textrm{d}\tau^2+2\varphi(\tau)\textrm{d}\theta^2.
\end{equation*}
Up to some constant, the distance between $\{\tau=1\}$ and $\{\tau=T\}$ is given by
\begin{align}
    \int_{1}^{T} \frac{1}{\sqrt{\varphi(\tau)}}\;\textrm{d}\tau&=\int_1^T \frac{\sqrt{\tau}\;\textrm{d}\tau}{\sqrt{\frac{1}{3}(\beta_1-1)(\tau-1)(\tau-\alpha_1)(\tau-T)}}\notag\\
    &=\frac{1}{\sqrt{\frac{1}{3}(1-\beta_1)}}\int_1^T\frac{\sqrt{\tau}\;\textrm{d}\tau}{\sqrt{(\tau-1)(\tau-\alpha_1)(T-\tau)}}\notag\\
    &\overset{\xi=\tau-1}{=\joinrel=}\frac{1}{\sqrt{\frac{1}{3}(1-\beta_1)}}\int_0^{T-1}\frac{\sqrt{\xi+1}\;\textrm{d}\xi}{\sqrt{\xi(\xi+1-\alpha_1)(T-1-\xi)}}\notag\\
    &=:\int_{0}^{T-1} I\;\textrm{d}\xi.\label{eq: estimate I}
\end{align}
Near $\xi=0$, terms $\sqrt{\xi+1}$ and $\sqrt{\xi+1-\alpha_1}$ are uniformly bounded as $\beta_1\to 1$. \eqref{eq: estimate I} satisfies
\begin{equation*}
\int_0^\epsilon I\;\textrm{d}\xi\leq C\cdot\frac{1}{\sqrt{1-\beta_1}}\cdot\frac{1}{\sqrt{\frac{1}{1-\beta_1}}}\cdot \sqrt{\xi}|_{0}^\epsilon,
\end{equation*}
for some uniform constant $C>0$ and any small $\epsilon>0$. Thus the integration in \eqref{eq: estimate I} does not blow up near $\xi=0$. Near $\xi=T-1$, in \eqref{eq: estimate I} for any fixed $\epsilon>0$, we can find $\beta_1$ close to $1$ such that for $\xi\in(T-1-\epsilon, T-1)$, we have
\begin{align*}
    &\sqrt{\xi+1}\geq \sqrt{T-\epsilon}\geq \frac{1}{2}T,\\
    &\sqrt{\xi(\xi+1-\alpha_1)}\leq \sqrt{(T-1)(T-\alpha_1)}\leq \sqrt{2}T.
\end{align*}
Then we have
\begin{align}
    \int_{T-1-\epsilon}^{T-1}I\;\textrm{d}\xi&\geq C\cdot\frac{1}{\sqrt{1-\beta_1}}\cdot \frac{1-\beta_1}{\sqrt{1-\beta_1}}\cdot \int_{T-1-\epsilon}^{T-1}\frac{\textrm{d}\xi}{\sqrt{T-1-\xi}}\notag\\
    &=C\cdot\frac{1}{\sqrt{1-\beta_1}}\cdot \frac{1-\beta_1}{\sqrt{1-\beta_1}}\cdot \sqrt{\xi}|_{0}^{\epsilon}\label{eq: near T-1}
\end{align}
for a uniform constant $C>0$ and arbitrary $\epsilon>0$. Since we can choose arbitrary large $\epsilon$ in \eqref{eq: near T-1}, the integration in \eqref{eq: estimate I} does not converge as $T\to \infty$. Combining the discussions above we see that $\int_0^{T-1}I\;\textrm{d}\xi$ tends to $\infty$ as $\beta_1\to 1$, i.e., $Z_{-2}$ gets pushed-off to infinity.
\end{proof}
From now on, we use $\beta_1$ as a subscript to emphasize the dependence on angles.
By \eqref{eq: eta origin expression} and \eqref{eq:varphi and tau}, on any compact subsets of $\mathbb{F}_2$ we have the following convergence in the $C^k$-norm for every $k$:
\begin{align}
    \eta_\infty:&=\lim_{\beta_1\to1}\eta_{\beta_1}\notag\\
    &=2\tau_{\beta_1} \frac{\sqrt{-1}\textrm{d}z\wedge\textrm{d}\bar{z}}{(1+|z|^2)^2}+\frac{1}{2}\left(\tau_{\beta_1}-\frac{1}{\tau_{\beta_1}}\right)\left(\frac{\sqrt{-1}\textrm{d}w\wedge\textrm{d}\bar{w}}{|w|^2}+\sqrt{-1}\alpha\wedge \bar{\alpha}\right.\notag\\
    &\left.+\sqrt{-1}\alpha\wedge \frac{\textrm{d}\bar{w}}{\bar{w}}+\sqrt{-1}\frac{\textrm{d}w}{w}\wedge \bar{\alpha}\right).\label{eq:convergence of eta}
\end{align}
Recall the Ricci curvature tensor of $\eta_{\beta_1}$ is given by
\begin{equation}\label{eq: ric tensor eta}
\begin{aligned}
    \operatorname{Ric}\eta_{\beta_1}&=(1-\beta_1)[C_1]+(1-\beta_2)[C_2]+2\frac{\sqrt{-1}\textrm{d}z\wedge\textrm{d}\bar{z}}{(1+|z|^2)^2}-\sqrt{-1}\partial\bar{\partial}\log\tau_{\beta_1}\\
    &-\sqrt{-1}\partial\bar{\partial}\log\varphi_{\beta_1}.
\end{aligned}
\end{equation}
The Ricci curvature $\lambda_{\beta_1}$ of $\eta_{\beta_1}$ is given by $\displaystyle \lambda_{\beta_1}=1-\beta_1$. Notice that $\lambda_{\beta_1}\to0$ as $\beta_1\to1$. Hence, by \eqref{eq: ric tensor eta} and facts that $\varphi\to(\tau^2-1)/2\tau$ in the limit and $\tau$ is a function of $r^4=|w|^2(1+|z|^2)^2$ (recall \eqref{eq: relation r u z}) we have
\begin{align}
    &2\frac{\sqrt{-1}\textrm{d}z\wedge\textrm{d}\bar{z}}{(1+|z|^2)^2}-\sqrt{-1}\partial\bar{\partial}\log\tau_{\infty}-\sqrt{-1}\partial\bar{\partial}\log\varphi_{\infty}=0,\notag\\
    \Rightarrow &\tau_\infty \varphi_\infty=C|w|^2(1+|z|^2)^2,\notag\\
    \Rightarrow&\tau_\infty=C^{-\frac{1}{2}}(C+r^4)^{\frac{1}{2}},\quad\textrm{for some constant}\;C>0.\label{eq: limiting tau}
\end{align}
Replacing $\tau$ and $\varphi$ in \eqref{eq: eta in limit} using \eqref{eq: limiting tau}, we have 
\begin{equation}\label{eq: eta infty}
    \eta_\infty=2C^{-\frac{1}{2}}\left((C+r^4)^{\frac{1}{2}}(\sigma_1^2+\sigma_2^2)+\frac{r^2}{(C+r^4)^{\frac{1}{2}}}(\textrm{d}r^2+r^2\sigma_3^2)\right).
\end{equation}
Comparing \eqref{eq: eta infty} to \eqref{eq: eh metric in r} we see $\eta$ converges on compact subsets to an Eguchi--Hanson metric as $\beta_1\to 1$. Summarizing discussions above, we have shown the following result. It is a special case of Theorem \ref{thm: metric limit nk case} and provides a new way of understanding Eguchi--Hanson metrics.

\begin{thm}\label{thm: limit n2 case}
Fix an arbitrary base point $p$ on the zero section. The K\"{a}hler--Einstein edge metric $(\mathbb{F}_2, \eta_{\beta_1}, p)$ on $\mathbb{F}_2$ converges in the pointed Gromov--Hausdorff sense 
to the following Eguchi--Hanson metric $(-2H_{\mathbb{P}^1}, \eta_\infty), p$ on $-2H_{\mathbb{P}^1}$ as $\beta_1\to 1$:
\begin{equation*}
    \eta_\infty=2C^{-\frac{1}{2}}\left((C+r^4)^{\frac{1}{2}}(\sigma_1^2+\sigma_2^2)+\frac{r^2}{(C+r^4)^{\frac{1}{2}}}(\textrm{d}r^2+r^2\sigma_3^2)\right),
\end{equation*}
where $C>0$ is a constant.
\end{thm}
\begin{proof}
The convergence on any compact subset of $-2H_{\mathbb{P}^1}$ of such $\eta$ to an Eguchi--Hanson metric in $C^k$-norm for every $k$ can be seen from \eqref{eq: eta infty}. Combining this fact, the fact that $Z_{-2}$ gets pushed-off to infinity as $\beta_1\to1$ and choosing an arbitrary base point from the exceptional divisor $Z_2$, we obtain the convergence in the pointed Gromov--Hausdorff sense by considering convergence of $\eta$ to an Eguchi--Hanson metric on compact geodesic balls centered at the base point.
\end{proof}

\section{Examples of limit of \KE edge metrics}\label{app: more eg}
In this section, we fix $k=1$ and $n=2$. Then we follow Section \ref{sec: asym of kee} and Section \ref{sec: gh kee} to give more concrete examples as limit of \KE edge metrics. In the previous work, we treated the case $\beta_1\searrow 0$ and $\beta_1 \to 1$ \cite{RZ21}. In this section, we consider the several cases: $\beta_1\nearrow 2$, $\beta_2 \nearrow 1$ with no rescaling and $\beta_2\nearrow 1$ with rescaling. The asymptotic behaviors in such cases are summarized in Table \ref{summary}.


Under the assumption $k=1$ and $n=2$, $\beta_1$ ranges from $(0, 2)$. $\tau$ ranges from $[1, T]$. We will study the limiting behaviors of \KE edge metrics when $\beta_1\to 2$. 

By \eqref{eq: varphi general expression nk case} $\varphi(\tau)$ satisfies
\begin{align*}
    \varphi(\tau)&=\frac{\tau^2-1}{\tau}+\frac{1}{3}(\beta_1-2)(\tau^3-1)\\
    &=\frac{1}{3}(\beta_1-2)(\tau-1)(\tau-\alpha_1)(
    \tau-T)/\tau,
\end{align*}
where $T$ and $\alpha_1$ satisfy \cite[(5.1), (5.2)]{RZ21}
\begin{align}
\begin{aligned}\label{eq: T a1 simcase}
    T&=1+3\frac{\sqrt{1+\frac{2}{3}\beta_1-\frac{1}{3}\beta_1^2}+\beta_1-1}{4-2\beta_1},\\
    \alpha_1&=1+3\frac{-\sqrt{1+\frac{2}{3}\beta_1-\frac{1}{3}\beta_1^2}+\beta_1-1}{4-2\beta_1}.
\end{aligned}
\end{align}
Obviously $T\to +\infty$ and $\alpha_1\to-1$ as $\beta_1\to2$. $\beta_2$ is given by \cite[(4.20)]{RZ21}
\begin{equation*}
    \beta_2=\frac{\beta_1-3+3\sqrt{1+\frac{2}{3}\beta_1-\frac{1}{3}\beta_1^2}}{2}.
\end{equation*}
Thus $\beta_2\to 1$ as $\beta_1\to 2$. More precisely, we have the following asymptotic behavior of $\beta_2(\beta_1)$:
\begin{lem}
For $\beta_1 < 2$ and close to $2$, $\beta_2(\beta_1)=\frac{1}{2}\beta_1+o(1)$.
\end{lem}
The length of the path on each fiber between the intersection point of the fiber with $Z_1$ and that of the fiber with $Z_{-1}$ is given by the integration of $\varphi(\tau)$ from $1$ to $T$. A similar calculation as in the proof of Proposition \ref{prop: length to infty n2 case} shows the following result.
\begin{prop}\label{prop: length infty n1 case}
The length of the path on each fiber between the intersection point of the fiber with $Z_1$ and that of the fiber with $Z_{-1}$ tends to infinity as $\beta_1\to 2$.
\end{prop}
Recall we assume K\"{a}hler--Einstein edge metrics on $\mathbb{F}_1$ have the form as in \eqref{eq: eta origin expression}. The Ricci curvature form of such \KE edge metrics, denoted by $\eta$, is given by \eqref{eq: ric tensor eta}. From now on, we denote by $\eta_{\beta_1}$, $\tau_{\beta_1}$ and $\varphi_{\beta_1}$ to emphasize the dependence of metrics and coordinates on $\beta_1$. Let us consider on any compact subsets of $\mathbb{F}_1$,
\begin{align*}
    \lim_{\beta_1\to 2}\eta_{\beta_1}&=\lim_{\beta_1\to 2}\tau\frac{\sqrt{-1}\textrm{d}z\wedge\textrm{d}\bar{z}}{(1+|z|^2)^2}+\left(\frac{\tau^2-1}{\tau}+\frac{1}{3}(\beta_1-2)(\tau^3-1)\right)\left(\frac{\sqrt{-1}\textrm{d}w\wedge\textrm{d}\bar{w}}{|w|^2}\right.\\
    &\left.+\sqrt{-1}\alpha\wedge\bar{\alpha}+\sqrt{-1}\alpha\wedge\frac{\textrm{d}\bar{w}}{\bar{w}}+\sqrt{-1}\frac{\textrm{d}w}{w}\wedge\bar{\alpha}\right)\\
    &=\tau\frac{\sqrt{-1}\textrm{d}z\wedge\textrm{d}\bar{z}}{(1+|z|^2)^2}+\frac{\tau^2-1}{\tau}\left(\frac{\sqrt{-1}\textrm{d}w\wedge\textrm{d}\bar{w}}{|w|^2}+\sqrt{-1}\alpha\wedge\bar{\alpha}+\sqrt{-1}\alpha\wedge\frac{\textrm{d}\bar{w}}{\bar{w}}\right.\\
    &\left.+\sqrt{-1}\frac{\textrm{d}w}{w}\wedge\bar{\alpha}\right).\\
    &=:\eta_\infty
\end{align*}
The Ricci curvature $\lambda_{\beta_1}$ of $\eta_{\beta_1}$ is given by $\lambda_{\beta_1}=2-\beta_1$. As $\beta_1\to2$, the Ricci curvature $\lambda_{\beta_1}$ tends to $0$. Thus, the limit metric $\eta_\infty$ has Ricci curvature $0$. Hence by \eqref{eq: ric tensor eta}, $\tau_\infty$ and $\varphi_\infty$ satisfy
\begin{align}
    &2\frac{\sqrt{-1}\textrm{d}z\wedge\textrm{d}\bar{z}}{(1+|z|^2)^2}-\sqrt{-1}\partial\bar{\partial}\log\tau_{\infty}-\sqrt{-1}\partial\bar{\partial}\log\varphi_{\infty}=0,\notag\\
    \Rightarrow &\tau_\infty \varphi_\infty=C|w|^4(1+|z|^2)^2,\notag\\
    \Rightarrow&\tau_\infty=(1+C|w|^4(1+|z|^2)^2)^{\frac{1}{2}},\quad\textrm{for some constant}\;C>0.\label{eq: tau infty n1 case}
\end{align}
Thus we obtain the following theorem, which is a special case of Theorem \ref{thm: metric limit nk case}.
\begin{thm}\label{thm: limit n2k1 case}
Fix an arbitrary base point $p$ on $Z_1\subset \mathbb{F}_1$. As $\beta_1\to2$, the  K\"{a}hler--Einstein edge metric $(\mathbb{F}_1, \eta_{\beta_1}, p)$ on $\mathbb{F}_1$ converges in the pointed Gromov--Hausdorff sense to a Ricci-flat metric $(-H_{\mathbb{P}^1} , \eta_\infty, p)$ on $-H_{\mathbb{P}^1}$ with conic singularity of angle $4\pi$ along $Z_1$.
\end{thm}
\begin{proof}
By \eqref{eq: tau infty n1 case}, the limit metric $\eta_\infty$ has the form
\begin{align*}
    \eta_\infty=&(1+C|w|^4(1+|z|^2)^2)^{\frac{1}{2}}\frac{\sqrt{-1}\textrm{d}z\wedge\textrm{d}\bar{z}}{(1+|z|^2)^2}\\
    &+\frac{C|w|^4(1+|z|^2)^2}{(1+C|w|^4(1+|z|^2)^2)^{\frac{1}{2}}}\left(\frac{\sqrt{-1}\textrm{d}w\wedge\textrm{d}\bar{w}}{|w|^2}+\sqrt{-1}\alpha\wedge\bar{\alpha}+\sqrt{-1}\alpha\wedge\frac{\textrm{d}\bar{w}}{\bar{w}}\right.\\
    &\left.+\sqrt{-1}\frac{\textrm{d}w}{w}\wedge\bar{\alpha}\right),
\end{align*}
for some constant $C>0$. Thus $\eta_\infty$ has edge singularity of angle $4\pi$ along $Z_1$. For a fixed base point on $Z_1$, Proposition \ref{prop: length infty n1 case} shows that $Z_{-1}$ gets pushed-off to infinity in the limit. The remaining proof is the same as that of Theorem \ref{thm: limit n2 case}.
\end{proof}
\subsection{Calculations in terms of $\beta_2$}
In this section, we choose a base point from the infinity section and then study the limiti behavior of the \KE edge metrics. In the language of Section \ref{sec: asym of kee}, we will consider the \KE edge metrics that are parametrized by $\beta_2$.

As in Section \ref{sec: asym of kee}, we consider $u:=1/w$ as the fiber coordinate. Then $\{u=0\}$ is the infinity section and $\{u=\infty\}$ is the zero section. We still define $s$ as follows:
\begin{equation*}
    s=\log(1+|z|^2)-\log|u|^2.
\end{equation*}
Then as before $\{s=-\infty\}$ still corresponds to the zero section while $\{s=+\infty\}$ corresponds to the infinity section.

Denote now by $\xi$ the \KE edge metric that we seek on $\mathbb{F}_1$. Assume $\xi=\sqrt{-1}\partial\bar{\partial} f(s)$ for some smooth function $f(s)$. As \eqref{eq: xi metric} we calculate 
\begin{equation*}
    \eta=\sqrt{-1}\partial\bar{\partial}f(s)=\tau\pi_1^*\omega_{\operatorname{FS}}+\varphi\left(\frac{\sqrt{-1}\textrm{d}u\wedge\textrm{d}\bar{u}}{|u|^2}+\alpha\wedge\bar{\alpha}-\alpha\wedge\frac{\textrm{d}\bar{u}}{\bar{u}}-\frac{\textrm{d}u}{u}\wedge\bar{\alpha}\right),
\end{equation*}
where $\alpha:=\bar{z}\textrm{d}z/1+|z|^2$, $\tau=f'(s)$ and $\varphi=f''(s)$ as before. As in Section \ref{sec: asym of kee}, after a renormalization of the metric we may assume
\begin{equation*}
    \sup f'(s)=1.
\end{equation*}
We also assume $\inf f'(s)=t$, for some $t>0$. In other words, $\tau$ ranges from $[t, 1]$. 

Now we calculate the Ricci curvature form $\eta$ using coordinates $u$ and $z$:
\begin{align*}
    \operatorname{Ric}\eta&=-\sqrt{-1}\partial\bar{\partial}\log\eta^2\\
    &=-\sqrt{-1}\partial\bar{\partial}\log \frac{\tau\varphi}{|u|^2(1+|z|^2)^2}\\
    &=(1-\beta_1)[Z_1]+(1-\beta_2)[Z_{-1}]+2\pi^*_1 \omega_{\operatorname{FS}}-\sqrt{-1}\partial\bar{\partial}\log\tau\\
    &-\sqrt{-1}\partial\bar{\partial}\log\varphi.
\end{align*}
Denote the Ricci curvature by $\mu$. The \KE edge equation
\begin{equation*}
    \operatorname{Ric}\eta=\mu \eta+(1-\beta_1)[Z_1]+(1-\beta_2)[Z_{-1}]
\end{equation*}
is equivalent to
\begin{equation}\label{eq: ode in b2 case}
    2-\varphi_{\tau}-\varphi/\tau=\mu\tau,\quad\tau\in[t, 1].
\end{equation}
Note that \eqref{eq: ode in b2 case} gives us the same ODE derived in \cite[(4.12)]{RZ21}.
Now let us determine boundary conditions satisfied by $\varphi(\tau)$. The same arguments in \cite[Proposition 3.3]{RZ21} give us that
\begin{equation}\label{eq: bd cond n1k1b2}
    \varphi(t)=\varphi(1)=0,\quad \varphi'(t)=\beta_1,\quad \varphi'(1)=-\beta_2.
\end{equation}
Plugging boundary conditions in \eqref{eq: ode in b2 case} implies that
\begin{equation*}
    \mu=2+\beta_2.
\end{equation*}
The solution to \eqref{eq: ode in b2 case} is
\begin{equation}\label{eq: sol n1k1b2}
    \varphi(\tau)=\frac{\tau^2-1}{\tau}+\frac{2+\beta_2}{3}\frac{1-\tau^3}{\tau}.
\end{equation}
Combining \eqref{eq: bd cond n1k1b2} and \eqref{eq: sol n1k1b2}, we obtain the dependence of $t$ and $\beta_1$ on $\beta_2$:
\begin{align}
\begin{aligned}\label{eq: tb1 simp case}
    t&=\frac{1-\beta_2+\sqrt{(\beta_2-1)(-3\beta_2-9)}}{2(2+\beta_2)},\\
    \beta_1&=\frac{3}{2}+\frac{1}{2}\beta_2-\frac{1}{2}\sqrt{(1-\beta_2)(3\beta_2+9)}.
\end{aligned}
\end{align}
Next, inspired by the result in Theorem \ref{thm: limit n2k1 case}, we study the limiting behavior of \KE edge metrics when $\beta_2$ tends to $1$. 

\begin{prop}\label{prop: length finite k1 case}
The length of the path on each fiber between the intersection point of the fiber with $Z_1$ and that of the fiber with $Z_{-1}$ converges to a finite number as $\beta_2\to 1$.
\end{prop}
\begin{proof}
As shown in Proposition \ref{prop: length to infty n2 case}, when restricted to the fiber $\{z=0\}$, the distance between $\{\tau=t\}$ and $\{\tau=1\}$ is given by
\begin{equation}\label{eq: integrad n2k1 b2 case}
    \int_t^1 \frac{1}{\sqrt{\varphi(\tau)}}\;\textrm{d}\tau=\int_t^1 \frac{\sqrt{\tau}}{\sqrt{(2+\beta_2)/3}\sqrt{(1-\tau)(\tau-\alpha_1)(\tau-t)}}\;\textrm{d}\tau.
\end{equation}
It remains to show the integral in \eqref{eq: integrad n2k1 b2 case} converges uniformly as $\beta_2\to 1$. Near $\tau=1$, $\sqrt{\tau}/(\tau-\alpha_1)(\tau-t)$ in \eqref{eq: integrad n2k1 b2 case} is uniformly bounded as $\beta_2\to 1$. Thus for $\epsilon>0$,
\begin{align*}
    &\int_{1-\epsilon}^1 \frac{\sqrt{\tau}}{\sqrt{(2+\beta_2)/3}\sqrt{(1-\tau)(\tau-\alpha_1)(\tau-t)}}\;\textrm{d}\tau\\
    &\leq C\int_{1-\epsilon}^1 \frac{1}{\sqrt{1-\tau}}\;\textrm{d}\tau<\infty
\end{align*}
for some uniform constant $C>0$. In other words, the integral \eqref{eq: integrad n2k1 b2 case} does not blow up near $\tau=1$. It remains to study \eqref{eq: integrad n2k1 b2 case} near $\tau=t$. We consider a change of coordinate $\xi:=\tau-t$. Then for $\epsilon>0$,
\begin{equation*}
    \int_t^{t+\epsilon} \frac{\sqrt{\tau}}{\sqrt{(2+\beta_2)/3}\sqrt{(1-\tau)(\tau-\alpha_1)(\tau-t)}}\;\textrm{d}\tau\leq C\int_0^\epsilon \frac{\sqrt{\xi+t}}{\sqrt{\xi(\xi+t-\alpha_1)}},
\end{equation*}
for some uniform constant $C>0$. Since
\begin{equation*}
    \lim_{\xi\to 0}\frac{\sqrt{\xi+t}}{\xi+t-\alpha_1}=\frac{\sqrt{t}}{\sqrt{t-\alpha_1}}\leq C,
\end{equation*}
for some uniform constant $C>0$ when $\beta_2\to 1$, we conclude that the integral \eqref{eq: integrad n2k1 b2 case} converges as $\int_0^\epsilon 1/\sqrt{\xi}$ near $\tau=t$. Hence we have finished the proof. 
\end{proof}

From now on, we denote by $\xi_{\beta_2}$, $\tau_{\beta_2}$ and $\varphi_{\beta_2}$ to indicate the dependence of metrics and coordinates on $\beta_2$.

\begin{thm}\label{thm: appthm2}
Fix an arbitrary base point $p$ on $Z_{-1}\subset \mathbb{F}_1$. As $\beta_2\to 1$, the \KE edge metric $(\mathbb{F}_1, \xi_{{\beta}_2}, p)$ on $\mathbb{F}_1$ converge in the pointed Gromov--Hausdorff sense to the Fubini--Study metric $(\mathbb{P}^2, \omega_{\mathrm{FS}},p)$.
We will show that $\xi_{\beta_2}$ converges pointwise smoothly to a degenerate metric tensor that is the pull-back of the Fubini--Study metric under the blow-up map on $\mathbb{F}_1$.
\end{thm}
\begin{proof}
Denote by $\xi_\infty$, $\tau_\infty$ and $\varphi_\infty$ the metric and coordinates in the limit when $\beta_2\to 1$. To find a relation between $\tau_\infty$, $\varphi_\infty$ and $s$, consider the ODE satisfied by $s$ and $\tau_{\beta_2}$:
\begin{equation}\label{eq: ODE n1k1b2}
    \frac{\textrm{d}s}{\textrm{d}\tau_{\beta_2}}=\frac{1}{\varphi_{\beta_2}}=\frac{\tau_{\beta_2}}{(\frac{2+\beta_2}{3})(1-\tau_{\beta_2}^3)+\tau_{\beta_2}^2-1}.
\end{equation}
Letting $\beta_2\to1$ in \eqref{eq: ODE n1k1b2}, we obtain (up to a constant that can be chosen to be $0$)
\begin{equation}\label{eq: ot r1}
    s=\log\frac{\tau_{\infty}}{1-\tau_\infty},\quad\tau_\infty\in(0, 1),
\end{equation}
where $\tau_\infty$ ranges from $(0, 1)$ since $t\to 0$ as $\beta_2\to 1$. Obviously there holds
\begin{equation}\label{eq: ot r2}
    \varphi_\infty=\tau_\infty(1-\tau_\infty).
\end{equation}
Recall $s=\log(1+|z|^2)-\log|u|^2$. Combining \eqref{eq: ot r1} and \eqref{eq: ot r2}, the limit metric $\xi_\infty$ has the form:
\begin{equation}\label{eq_XiInfN2K1}
\begin{aligned}
    \xi_\infty&=\tau_\infty \pi_1^*\omega_{\operatorname{FS}}+\varphi_\infty \left(\pi_2^*\omega_{\operatorname{Cyl}}+\sqrt{-1}\alpha\wedge\bar{\alpha}-\sqrt{-1}\alpha\wedge\frac{\textrm{d}\bar{u}}{\bar{u}}-\sqrt{-1}\frac{\textrm{d}u}{u}\wedge\bar{\alpha}\right)\\
    &=\frac{1+|z|^2}{|u|^2+1+|z|^2}\pi_1^*\omega_{\operatorname{FS}}+\frac{1+|z|^2}{(|u|^2+1+|z|^2)^2}\sqrt{-1}\textrm{d}u\wedge\textrm{d}\bar{u}\\
    &+\frac{|u|^2(1+|z|^2)}{(|u|^2+1+|z|^2)^2}\left(\sqrt{-1}\alpha\wedge\bar{\alpha}-\sqrt{-1}\alpha\wedge\frac{\textrm{d}\bar{u}}{\bar{u}}-\sqrt{-1}\frac{\textrm{d}u}{u}\wedge\bar{\alpha}\right).
\end{aligned}
\end{equation}

Next, we derive the explicit formula for $\xi_\infty$. Recall,
\begin{equation}\label{eq_AlpDefN2K1}
    \alpha = \frac{\bar{z}\textrm{d}z}{1+|z|^2}.
\end{equation}

Plugging \eqref{eq_AlpDefN2K1} in \eqref{eq_XiInfN2K1} and by calculations:
\begin{equation}\label{eq_XiInfForm}
\begin{aligned}
    \xi_\infty &=\frac{1+|z|^2}{|u|^2+1+|z|^2} \cdot \frac{\sqrt{-1}\textrm{d}z\wedge \textrm{d}\bar{z}}{(1+|z|^2)^2} +\frac{1+|z|^2}{(|u|^2+1+|z|^2)^2}\sqrt{-1}\textrm{d}u\wedge\textrm{d}\bar{u}\\
    &+\frac{|u|^2(1+|z|^2)}{(|u|^2+1+|z|^2)^2}\left(\sqrt{-1} \cdot \frac{|z|^2\textrm{d}z\wedge \textrm{d}\bar{z}}{(1+|z|^2)^2}
    -\sqrt{-1}\frac{\bar{z}\textrm{d}z}{1+|z|^2}\wedge\frac{\textrm{d}\bar{u}}{\bar{u}}-\sqrt{-1}\frac{\textrm{d}u}{u}\wedge\frac{{z}\textrm{d}\bar{z}}{1+|z|^2}\right)\\
    & = \sqrt{-1}\left(\frac{1+|u|^2}{(1+|u|^2+|z|^2)^2} \textrm{d}z\wedge \textrm{d}\bar{z} - \frac{\bar{z}u}{(1+|u|^2+|z|^2)^2}\textrm{d}z\wedge \textrm{d}\bar{u} - \frac{\bar{u}z}{(1+|u|^2+|z|^2)^2}\textrm{d}u\wedge \textrm{d}\bar{z}\right. \\
    & + \left.\frac{1+|z|^2}{(1+|u|^2+|z|^2)^2}\textrm{d}u\wedge \textrm{d}\bar{u}
    \right)\\
    &=\sqrt{-1}\partial\bar{\partial} \log (1+|u|^2+|z|^2).
\end{aligned}
\end{equation}

This limit metric does not have singularity along $Z_-1$ since $\beta_2\to 1$ in the limit. Moreover, $\xi_\infty$ has Ricci curvature $3$ since $\mu_{\beta_2}$ tends to $3$ when $\beta_2\to 1$. Moreover, the metric $\xi_\infty$ degenerates along $Z_1 = \{u=\infty\}$ as $\xi_\infty\to 0$ as $|u|\to +\infty$.

We next show that $\xi_\infty$ is the pull-back of the Fubini--Study metric $\omega_{\operatorname{FS}}$ on $\mathbb{P}^2$ under the blow-down map $\pi: \mathbb{F}_1\to \mathbb{P}^2$. Figure \ref{fig_BlowUp} shows the blow up of $\mathbb{P}^2(1, 1, k)$ at $p$ giving rise to $\mathbb{F}_1$. To see this, we regard $\mathbb{F}_1$, which is the blow-up of $\mathbb{P}^2$ at one point $p$ (WLOG assuming $p=[1:0:0]$), as the variety embedded in $\mathbb{P}^2\times \mathbb{P}^1$:
\begin{equation*}
    \mathbb{F}_1 = \{([x_0:x_1:x_2], [y_0:y_1])\in \mathbb{P}^2\times \mathbb{P}^1: x_1 y_1 = x_2 y_0\}.
\end{equation*}
Then the blow-down map $\pi$ is given by:
\begin{equation*}
\begin{aligned}
    \pi: \mathbb{F}_1 &\to \mathbb{P}^2\\
    ([x_0:x_1:x_2], [y_0:y_1])&\mapsto [x_0:x_1:x_2].
 \end{aligned}   
\end{equation*}
Restricted on the chart $\{x_1\neq 0\}$, there hold:
\begin{equation*}
    u = \frac{x_0}{x_1}, \quad z = \frac{x_2}{x_1}
\end{equation*}
and 
\begin{equation}\label{eq_PiForm}
    \pi(u, z) = (u, z)\in\mathbb{P}^2.
\end{equation}
Recall the Fubini--Study metric on $\mathbb{P}^2$ (restriced to the chart $\{x_1\neq 0\}$) has the formula
\begin{equation*}
    \omega_{\operatorname{FS}} = \sqrt{-1}\log(1+|u|^2+|z|^2).
\end{equation*}
Thus by \eqref{eq_XiInfForm} and \eqref{eq_PiForm} we have shown,
\begin{equation*}
    \xi_\infty = \pi^* \omega_{\operatorname{FS}},
\end{equation*}
confirming that $\xi_\infty$ degenerates along the exceptional curve $Z_1$.

We have shown $\xi_{\beta_2}$ converges on any compact subset of $\mathbb{F}_1$ to $\xi_\infty$. Now fix an arbitrary base point on $Z_{-1}\subset\mathbb{F}_1$. By Proposition \ref{prop: length finite k1 case}, the length between $Z_1$ and $Z_{-1}$ remains to be finite when $\beta_2\to 1$. Thus $(\mathbb{F}_1, \xi_{\beta_2}, p)$ converges in the pointed Gromov--Hausdorff sense to $(\mathbb{P}^2, \omega_{\mathrm{FS}}, p)$ when $\beta_2\to 1$. 
\end{proof}

\begin{figure}[!htbp]
    \centering
    \includegraphics[width=10cm]{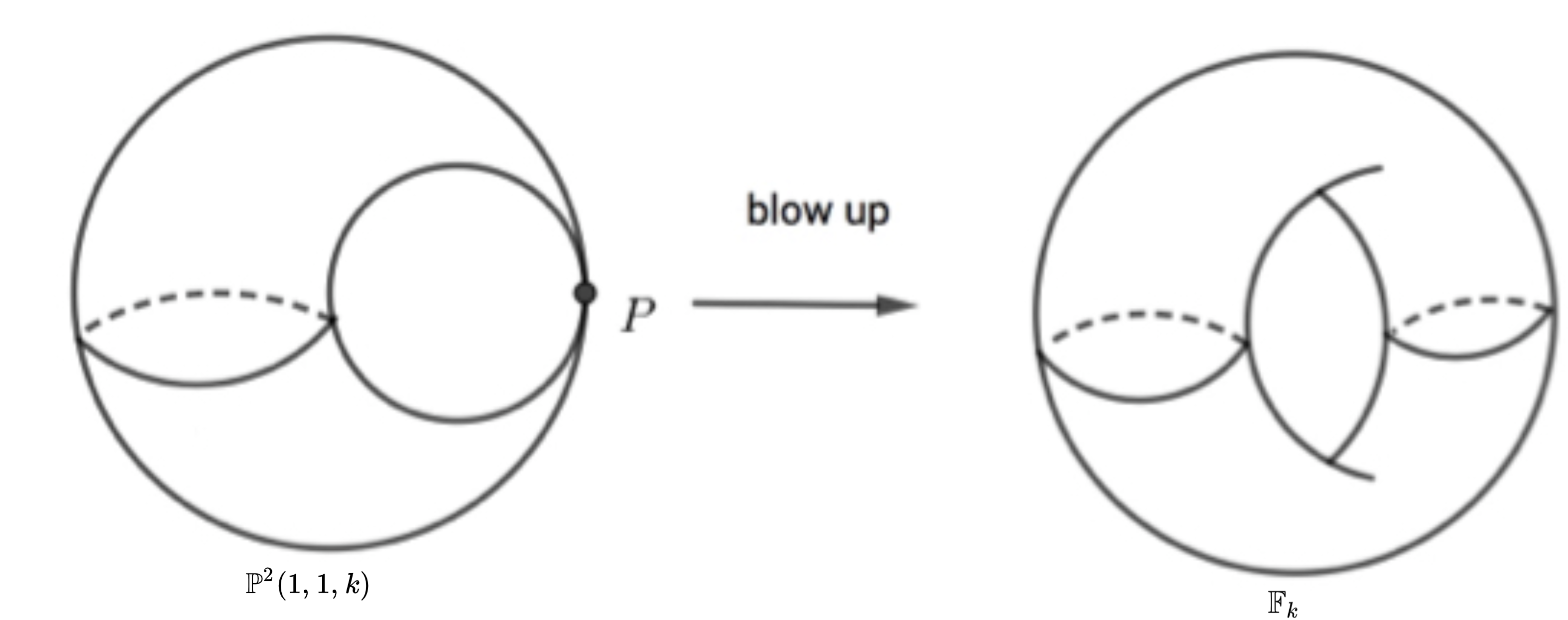}
    \caption{Blow up of $\mathbb{P}^2(1,1,k)$ at $p$ (with $k>1$).}
    \label{fig_BlowUp}
\end{figure}

Next, we consider the limiting behavior of properly renormalized \KE edge metrics on $\mathbb{F}_1$ when $\beta_2 \to 1$.

\begin{lem}\label{lem: two t relations simcase}
Rescaling the metric $\xi$ by a factor $2(2+\beta_2)/(1-\beta_2+\sqrt{(1-\beta_2)(3\beta_2+9)})$, the interval of definition of $\tau$ will change from $[t, 1]$ to $[1, T]$ as in \eqref{eq: T a1 simcase}.
\end{lem}
\begin{proof}
When we calculate in terms of $\beta_1$, we have
\begin{align}
    \beta_2&=\frac{\beta_1-3+3\sqrt{1+\frac{2}{3}\beta_1-\frac{1}{3}\beta_1^2}}{2},\label{eq: b2 compare}\\
    T&=1+3\frac{\sqrt{1+\frac{2}{3}\beta_1-\frac{1}{3}\beta_1^2}+\beta_1-1}{4-2\beta_1}.\label{eq: big t compare}
\end{align}
When we calculate in terms of $\beta_2$, we have
\begin{align}
    \beta_1&=\frac{3}{2}+\frac{1}{2}\beta_2-\frac{1}{2}\sqrt{(\beta_2-1)(-3\beta_2-9)},\label{eq: b1 compare}\\
    t&=\frac{1-\beta_2+\sqrt{(\beta_2-1)(-3\beta_2-9)}}{2(2+\beta_2)}.\label{eq: lil t compare}
\end{align}
Direct calculation shows \eqref{eq: b2 compare} and \eqref{eq: b1 compare} are equivalent. Combining \eqref{eq: b2 compare} and \eqref{eq: big t compare}, we have
\begin{equation*}
    T=\frac{2(2+\beta_2)}{4-2\beta_1}.
\end{equation*}
Combining \eqref{eq: b1 compare} and \eqref{eq: lil t compare}, we have
\begin{equation*}
    t=\frac{4-2\beta_1}{2(2+\beta_2)}.
\end{equation*}
Thus, $T=1/t$ and we can rescale the metric $\eta$ by the factor $1/t$ to change the domain of $\tau$ from $[t, 1]$ to $[1, T]$.
\end{proof}

Inspired by Lemma \ref{lem: two t relations simcase}, we normalize $\xi_{\beta_2}$ by the factor $\sqrt{3}/\sqrt{1-\beta_2}$ and study its limiting behavior when $\beta_2$ tends to $1$.

\begin{thm}\label{thm: rescale simcase}
Rescaling the metric $\xi_{\beta_2}$ by $\sqrt{3}/\sqrt{1-\beta_2}$, then as $\beta_2\to 1$, the renormalized metric converges in the pointed Gromov--Hausdorff sense to a Ricci-flat metric on $-H_{\mathbb{P}^1}$, where the base point is chosen from $Z_1$. See the proof for an explicit explanation. This metric coincides with the one obtained in Theorem \ref{thm: limit n2k1 case}.
\end{thm}
\begin{proof}
Consider the change of coordinate $y_{\beta_2}:=\sqrt{3}\tau_{\beta_2}/\sqrt{1-\beta_2}$ in the following ODE:
\begin{equation*}
    \frac{\textrm{d}s}{\textrm{d}\tau_{\beta_2}}=\frac{1}{\varphi_{\beta_2}(\tau_{\beta_2})},
\end{equation*}
In the following equations, we omit the subscript $\beta_2$:
\begin{align*}
    \frac{\textrm{d}s}{\textrm{d}y}\frac{\textrm{d}y}{\textrm{d}\tau}&=\frac{1}{\varphi(y)}\\
    \Rightarrow \frac{\textrm{d}s}{\textrm{d}y}&=\frac{\sqrt{3}}{3}\cdot\frac{(1-\beta_2)y}{(\sqrt{1-\beta_2}y)^2-1+\frac{2+\beta_2}{3}(1-(\sqrt{1-\beta_2}y)^3)}.
\end{align*}
By an abuse of notation, still denote by $y$ the coordinate in the limit. As $\beta_2\to 1$, there holds
\begin{equation}\label{eq: sy mid case}
    \frac{\textrm{d}s}{\textrm{d}y}=\frac{y}{y^2-1},\quad y\in\left(1, +\infty\right),
\end{equation}
where the range of $y$ comes from \eqref{eq: lil t compare}. Recall the renormalized metric $\sqrt{3}\xi_{\beta_2}/\sqrt{1-\beta_2}$ reads
\begin{equation}\label{eq: metric mid case}
\begin{aligned}
    \frac{\sqrt{3}\xi_{\beta_2}}{\sqrt{1-\beta_2}}&=\frac{\sqrt{3}\tau_{\beta_2}}{\sqrt{1-\beta_2}}\pi_1^*\omega_{\operatorname{FS}}+\frac{\sqrt{3}\varphi_{\beta_2}}{\sqrt{1-\beta_2}}\left(\pi_2^*\omega_{\operatorname{Cyl}}+\sqrt{-1}\alpha\wedge\bar{\alpha}-\sqrt{-1}\alpha\wedge\frac{\textrm{d}\bar{u}}{\bar{u}}\right.\\
    &\left.-\sqrt{-1}\frac{\textrm{d}{u}}{{u}}\wedge\alpha\right).
\end{aligned}
\end{equation}
Combining \eqref{eq: sy mid case} and \eqref{eq: metric mid case} we obtain the limit metric
\begin{equation*}
\begin{aligned}
    \tilde{\xi}_\infty&=\sqrt{e^{2s}+1}\pi_1^*\omega_{\operatorname{FS}}+\frac{e^{2s}}{\sqrt{e^{2s}+1}}\left(\pi_2^*\omega_{\operatorname{Cyl}}+\sqrt{-1}\alpha\wedge\bar{\alpha}-\sqrt{-1}\alpha\wedge\frac{\textrm{d}\bar{u}}{\bar{u}}\right.\\
    &\left.-\sqrt{-1}\frac{\textrm{d}{u}}{{u}}\wedge\alpha\right).
\end{aligned}
\end{equation*}
Note that this limit metric is Ricci-flat. It coincides with the metric obatined in Theorem \ref{thm: limit n2k1 case} if we choose $C=1$ there. We have shown the renormalized \KE edge metrics converge to $\tilde{\xi}_\infty$ in smooth local sense. Next, fix a base point from the zero section $Z_1$. Then by Proposition \ref{prop: length finite k1 case}, we conclude that the infinity section $Z_{-1}$ gets pushed-off to infinity in the limit. Combining this fact with the local smooth convergence, we conclude the pointed Gromov--Hausdorff limit of $(\mathbb{F}_1, \sqrt{3}\eta_{\xi_2}/\sqrt{1-\beta_2})$ is $(-H_{\mathbb{P}^1}, \tilde{\xi}_\infty)$.
\end{proof}

\bigskip
{\sc University of Maryland}

{\tt yxji@umd.edu, yanir@alum.mit.edu}

\bigskip

{\sc Laboratory of Mathematics and Complex Systems, School of Mathematical Sciences, Beijing Normal University, Beijing, 100875, P. R. China.}

{\tt kwzhang@bnu.edu.cn}


\end{document}